\documentclass[11pt,superscriptaddress]{revtex4-1}
\usepackage[utf8]{inputenc}
\usepackage[T1]{fontenc}
\usepackage{amssymb}
\usepackage{amsmath}
\usepackage{amsthm}
\usepackage[colorlinks=false,pdfborder={0 0 0}]{hyperref}
\usepackage{url}

\begin{document}

\newtheorem{Def}{Definition}[section]
\newtheorem{Prop}[Def]{Proposition}
\newtheorem{Lem}[Def]{Lemma}
\newtheorem{Cor}[Def]{Corollary}
\newtheorem{Thm}[Def]{Theorem}
\newtheorem{remark}[Def]{Remark}

\title{Type II blow-up mechanism for supercritical harmonic map heat flow}

\begin{abstract}
  The harmonic map heat flow is a geometric flow well known to produce
  solutions whose gradient blows up in finite time.  A popular model
  for investigating the blow-up is the heat flow for maps
  $\mathbb R^{d}\to S^{d}$, restricted to equivariant maps.  This
  model displays a variety of possible blow-up mechanisms, examples
  include self-similar solutions for $3\le d\le 6$ and a so-called
  Type II blow-up in the critical dimension $d=2$.  Here we present
  the first constructive example of Type II blow-up in higher
  dimensions: for each $d\ge7$ we construct a countable family of Type
  II solutions, each characterized by a different blow-up rate.  We
  study the mechanism behind the formation of these singular solutions
  and we relate the blow-up to eigenvalues associated to linearization
  of the harmonic map heat flow around the equatorial map.  Some of
  the solutions constructed by us were already observed numerically.
\end{abstract}

\author{Paweł Biernat}
\email{pawel.biernat@gmail.com}
\affiliation{Institute of Mathematics, University of Bonn.}
\author{Yukihiro Seki}
\email{seki@math.kyushu-u.ac.jp}
\affiliation{Department of Mathematical Sciences, Faculty of Mathematics, Kyushu University.}

\maketitle

\section{Introduction}
Given a map $F:M\to N\subset \mathbb R^k$ between two Riemannian manifolds
$M$ and $N$ one can define a functional
\begin{align}
  \label{eq:49}
  E(F)=\frac{1}{2}\int_M |\nabla F|^2\,dV_M.
\end{align}
The critical points of this functional are referred to as harmonic
maps and one of the well established procedures used to prove the
existence of such objects, called harmonic map flow, was introduced by
\cite{Eells1964} in the 50's.  The basic intuition behind the harmonic
map flow is that it is a gradient flow for the functional
\eqref{eq:49}.  The partial differential equation governing the
harmonic map flow can be concisely expressed as
\begin{align}
  \label{eq:80}
    \partial_t F_t=(\Delta F_t)^\top,\quad F_t|_{t=0}=F_0,
\end{align}
where $v^\top$ is a projection of $v\in \mathbb R^k$ to $T_{F_t}N$ and
$F_{0}$ is some initial map $F_{0}:M\to N$.  As long as a solution to
\eqref{eq:80} stays smooth the value of $E(F_{t})$ is decreasing,
unless $F_{t}$ is a stationary point of $E$, that is a harmonic
map.  If the solution stays smooth for all times one can recover the
homotopy between the initial map $F_{0}$ and the limit $F_{\infty}$,
granted it exists, thus proving that there exists a harmonic map in
the homotopy class of $F_{0}$.  This is indeed true if manifold $N$
has everywhere nonpositive sectional curvature, however the claim is
false in general.

As in other geometric flows, what prevents the solutions from existing
in the long run, even when starting from smooth initial data, is the
onset of singularities.  In case of the harmonic map flow these
singularities take the form of discontinuities in $F_{t}$, meaning
that right before the singularity the quantity $\lvert\nabla F\rvert$
blows up.  One could cope with the singularities by introducing some
weak notion of solutions but with the loss of continuity of $F_{t}$ we
also lose the interpretation of $F_{t}$ as a homotopy.  Moreover,
there are examples of weak solutions to harmonic map flow, which are
smooth before and after the formation of the singularity but lose
uniqueness upon transitioning through the singularity.  We believe
that by studying the mechanisms behind the formation of singularities
we might be able to shed some light on this loss of uniqueness.  The
first step in such analysis is to recover the structure of singular
solutions, right before the singularity occurs.

In this paper we consider only a class of maps
$F_t:\mathbb R^d\to S^d$ (note that $S^{d}$ is the classical choice
for a simple positively curved manifold) for which equation
\eqref{eq:80} takes a form
\begin{align}
  \label{eq:83}
    \partial_t F_t=\Delta F_t+|\nabla F|^2 F_t.
\end{align}
Under further restriction of $F_t$ to a $1$-equivariant map
\begin{align}
  \label{eq:84}
  F_t(r,\omega)=(\cos(u(r,t)),\sin(u(r,t))\omega),\qquad \omega\in S^{d-1}
\end{align}
we arrive at
\begin{align}
  \label{eq:u}
  \partial_t u=\frac{1}{r^{d-1}}\partial_r\left(r^{d-1}\partial_r
    u\right)-\frac{d-1}{2r^2}\sin(2u).
\end{align}
The boundary conditions depend on the specific formulation of the
Cauchy problem for \eqref{eq:u}, for example demanding that the energy
density $\lvert \nabla F_{t}\rvert$ is finite yields $u(0,t)=0$ and
$\lim_{r\to\infty}u(r,t)=n\pi$ while $\partial_{r}u(r,t)$ satisfies a
certain growth bound at spatial inifinity.  In the next section we
discuss the conditions for the well-posedness of \eqref{eq:u} in
detail and show that \eqref{eq:u} yields smooth solutions under the
sole assumption of $u_{0}(r)/r$ being bounded (although these
solutions do not have, in general, finite initial energy
density).

The blow-up in \eqref{eq:u} can only occur at $r=0$ and it manifests
itself as
\begin{align*}
  \lim_{t\to T}|\partial_r u(0,t)|= \infty
\end{align*}
for some $T>0$; likewise the spatial scale associated to the blow-up
is given by
\begin{align*}
  R(t):=\frac{1}{\lvert\partial_r u(0,t)\rvert}.
\end{align*}
It is important to distinguish between two cases, depending on how
quickly the gradient tends to infinity: if
\begin{align*}
  \sup_{t<T}\lvert\sqrt{T-t}\,\partial_{r}u(0,t)\rvert<\infty,
\end{align*}
we say that the blow-up is of Type I, otherwise we call it a Type II
blow-up.  This distinction is firmly connected to the underlying
blow-up mechanism, in case of Type I blow-up the solutions are
asymptotically, as $t\to T$, self-similar and of the form
$u(r,t)=f(r/\sqrt{T-t})$.  Conversely, for Type II blow-up the
singularity must have a harmonic map at its core,
i.e. $u(r,t)=U(r/R(t))$ for $r=\mathcal O(R(t))$ with $U$ being the
stationary solution to \eqref{eq:u}.  These two blow-up types comprise
all possible blow-up scenarios, meaning that there are no blow-up
solutions with $\sqrt{T-t}/R(t)\to 0$ as $t\to T$.

Despite a relatively simple formulation, the problem \eqref{eq:u}
admits a wide range of possible blow-up mechanisms depending on
dimension $d$.  For $d=2$, the generic blow-up is of Type II and it is
realized by a shrinking harmonic map containing a finite energy so the
blow-up may be viewed as ``bubbling'' process, where some portion of
energy is trapped inside the singularity (see \cite{VandenBerg2003}
for a formal approach and \cite{Raphael2011} for a proof).  In higher
dimensions, $3\le d\le 6$, there exists a family
$(f_{n}(r/\sqrt{T-t}))_{n=0,1,\dots}$ of self-similar solutions to
\eqref{eq:u} \cite{Fan1999}, with $f_{0}$ corresponding to the generic
blow-up and each next $f_{n}$ being less stable than the previous one
\cite{Biernat2011}, and each corresponding to a Type I blow-up.  It is
also worth noting that these self-similar solutions can be used to
construct the non-unique weak solutions to \eqref{eq:u}
\cite{Biernat2011}, \cite{Germain2010}.  So far, the existence of Type
II blow-up in dimensions $3\le d\le 6$ has not been ruled out.

In even higher dimensions, $d\ge7$, there can be no Type I blow-up
(Bizoń and Wasserman \cite{Bizon2014} proved that there are no
self-similar solutions to \eqref{eq:u}).  Instead in
\cite{Biernat2014} one of the authors constructed (non-rigorously) a
countable family $(u_{n})_{n=0,1,\dots}$ of Type II solutions.  This
family of solutions display a remarkably similar stability hierarchy
to its lower-dimensional self-similar counterpart
$(f_{n}(r/\sqrt{T-t}))_{n=0,1,\dots}$, with $u_{0}$ corresponding to a
generic blow-up, $u_{1}$ being the co-dimension one solution and so
on.  Here we mimic this construction in a rigorous fashion and prove
that representatives of these solutions actually exist.

\begin{Thm}  \label{th:main-rate}
  Fix $d\ge8$ and $l\ge1$ or $d=7$ and $l\ge2$ and define a real number $\gamma$ as
  \begin{equation*}
    \gamma=\frac{1}{2}(d-2-\sqrt{d^{2}-8d+8}).
  \end{equation*}
  Then there exists a smooth solution $u_{l}$ to \eqref{eq:u} blowing
  up at time $T$, which is characterized by a blow-up rate
  \begin{align}
    \label{eq:81}
    R(t)=\frac{1}{\lvert\partial_r u_{l}(0,t)\rvert}=
(\alpha_{l}+o(1))(T-t)^{\frac{l}{\gamma}}
\quad \text{ as } t\to T
  \end{align}
  with some constant $\alpha_{l}>0$.  For given $d$ and $l$ there
  holds $\frac{l}{\gamma}>\frac{1}{2}$, that is in each case the
  blow-up is of Type II.
\end{Thm}

\begin{remark}
  The peculiar looking condition ``$d\ge8$ and $l\ge1$ or $d=7$ and
  $l\ge2$'' becomes more telling if one writes it as ``for any $d\ge7$
  pick $\lambda_{l}>0$'', where $\{\lambda_{l}\}$ are eigenvalues of a
  certain self-adjoint operator introduced later on.  The reason for
  the two cases $d\ge8$ and $d=7$ is that in $d=7$ the eigenvalue
  $\lambda_{1}$ is equal to zero.
\end{remark}

\begin{remark}
  Theorem \ref{th:main-rate} is strictly speaking an existence result
  and we do not claim anything about the blow-up rate of generic
  solutions nor about the stability of the constructed
  solutions.  There is, however, strong numerical evidence
  \cite{Biernat2014} that the solutions constructed by us for $d\ge8$
  and $l=1$ correspond to the generic blow-up scenario.
\end{remark}

\begin{remark}
  Our proof bases on the matched asymptotics method combined with a
  priri estimates and topological arguments, first developed by
  Herrero and Vel\'azquez
  \cite{herrero-velazquez-unpublished}.  Later, this method was
  applied to determine the blow-up rates for several parabolic
  problems with singularity formation; examples include two-dimensional chemotaxis
  model \cite{Herrero1996, Vela2002}, the problem of ice ball melting
  \cite{Herrero1997}, the problem of dead core formation   \cite{guo2008finite, Seki2011}.  The main advantage of this method is that it
  directly relates the blow-up rate to spectral information of the linearized
  operator governing the evolution near the equatorial map
  $u(r,t)=\pi/2$, thus giving a clear qualitative description of the
  singular solution in the whole domain $r\ge 0$ (see
  Theorem~\ref{lem:intermediate-implications}).
\end{remark}

\section{Formal derivation of the blow-up rate}
\label{sec:derivation-blow-up}
In this section we briefly review the steps from \cite{Biernat2014}
leading to a construction of the family of nonrigorous approximate
solutions to \eqref{eq:u}.  Let us define the similarity variables
\begin{align}
  \label{eq:7}
  y=\frac{r}{\sqrt{T-t}},\qquad s=-\log(T-t), \qquad f(y,s)=u(r,t).
\end{align}
In similarity variables equation \eqref{eq:u} becomes
\begin{align}
  \label{eq:f}
  \partial_{s}f=\partial_{yy}f+\left(\frac{d-1}{y}-\frac{y}{2}\right)\partial_{y}f-\frac{d-1}{2y^{2}}\sin(2f),\qquad f(0,s)=0.
\end{align}
with the boundary condition $f(0,s)=0$ coming from
$f(0,s)=u(0,t)=0$.  Numerical experiments suggest that singular
solutions to \eqref{eq:u} correspond to solutions of \eqref{eq:77}
converging to $f(y,s)=\frac{\pi}{2}$.  So as a first step we linearize
the equation \eqref{eq:77} around $\frac{\pi}{2}$ by defining a
variable $\psi$ such that $f=\frac{\pi}{2}+\psi$.  Let us formally
write the equation for $\psi$ as
\begin{align}
  \label{eq:ss}
  \partial_{s}\psi&=-\mathcal A\psi+F(\psi)\\
  F(\phi)(y)&=\frac{d-1}{2y^{2}}(\sin(2\phi(y))-(2\phi(y))),
\end{align}
with the linear part represented by the formal operator $\mathcal A$,
given by
\begin{align}
  -\mathcal A\phi(y)=\frac{1}{\rho}\frac{d}{dy}\left(\rho \frac{d\phi}{dy}(y)\right)+\frac{d-1}{y^{2}}\phi(y),\qquad\rho=y^{d-1}e^{-\frac{y^{2}}{4}}.
\end{align}
For $d\ge7$, the (formal) operator $\mathcal A$ can be uniquely
extended to a proper semibounded self-adjoint operator $A$ acting in
some Hilbert space $\mathcal H$ but the details of the functional
setup are inessential for the nonrigorous approach, so let us skip
them for now.  We then pick a positive eigenvalue of $A$, denoted as
$\lambda_{l}>0$, and we define a single-mode approximation to
\eqref{eq:f} as
\begin{align}
  \label{eq:74}
  f_{out}(y,s):=\frac{\pi}{2}+a_{l}(0)e^{-\lambda_{l}s}\phi_{l}
\end{align}
with $\phi_{l}$ being the corresponding eigenvector of $A$, while
$a_{l}(0)$ is an arbitrary real coefficient (in fact, the matching
condition \ref{eq:eps} below restricts $a_{l}(0)<0$).  It is easy to
see that $\psi(s,y)=f_{out}(y,s)-\frac{\pi}{2}$ solves the linear part
of \eqref{eq:ss} so it makes for a reasonable first-order
approximation.  Unfortunately, the eigenfunctions $\{\phi_{n}\}$ are
singular at the origin: $\phi_{n}(y)\sim y^{-\gamma}$ for some
$\gamma>0$, so $f_{out}(y,s)$ fails to satisfy the boundary condition
at the origin $f(0,t)=0$.  The divergence of $\phi_{n}$ at the origin
is a consequence of us linearizing the equation \eqref{eq:f} around a
``solution'' $f(y,s)=\frac{\pi}{2}$, which fails to satisfy the
boundary condition at the origin by itself.

If we consider the region of validity of the linear approximation to
be defined by $\lvert f_{out}(s)-\frac{\pi}{2}\rvert\ll 1$, which is equivalent to
$y\gg e^{-\frac{\lambda_{l}}{\gamma} s}$, then we get
\begin{align}
  \label{eq:76}
  f(y,s)\approx f_{out}(y,s)=\frac{\pi}{2}+a_{l}(0)e^{-\lambda_{l}s}\phi_{l},\qquad e^{-\frac{\lambda_{l}}{\gamma} s}\ll y.
\end{align}
Note that the region of validity of $\psi_{out}$ expands towards zero,
which implies an existence of a shrinking boundary layer, where the
true solution rapidly transitions from the boundary value $f(0,s)=0$
to $f(y,s)\approx\frac{\pi}{2}$.

To describe this boundary layer we introduce a natural spatial-scale
given by the gradient of $f(y,s)$ at the origin
\begin{align*}
  \xi=\frac{y}{\varepsilon(s)},\qquad\varepsilon(s):=\frac{1}{\lvert\partial_{y}f(0,s)\rvert},
\end{align*}
so $\varepsilon(s)$ roughly corresponds to the (so far unknown) width
of the boundary layer.  Let us also define a new variable
$U(\xi,s):=f(r,t)$, satisfying
\begin{align*}
\varepsilon^{2}\partial_{s}U=\partial_{\xi\xi}U+\left(\frac{d-1}{\xi}+(2\dot\varepsilon\varepsilon-\varepsilon^{2})\frac{\xi}{2}\right)\partial_{\xi}U-\frac{(d+1)}{2\xi^{2}}\sin(2U),\qquad U(0,s)=0.
\end{align*}
Anticipating $\varepsilon(s)\to 0$ as $s\to \infty$ we drop the
corresponding terms and arrive at
\begin{align}
  \label{eq:11}
  0=\partial_{\xi\xi}U+\frac{d-1}{\xi}\partial_{\xi}U-\frac{(d+1)}{2\xi^{2}}\sin(2U),\qquad U(0,s)=0
\end{align}
which is effectively an ordinary differential equation (with a unique
solution since we fixed the derivative at the origin in the definition
of $\varepsilon$ and $\xi$:
$\partial_{\xi}U(0,s)=1$).  Leading to an approximation via stationary
solution $U(\xi,s)\approx U_{1}(\xi)$ and consequently to
\begin{align}
  \label{eq:78}
  f(y,s)\approx f_{inn}(y,s):=U_{1}\left(\frac{y}{\varepsilon(s)}\right).
\end{align}
This time the approximation \eqref{eq:78} satisfies the boundary
condition at the origin.  However, the approximation by $f_{inn}$ is
also of a limited scope as it works only for
$\varepsilon^{2}\xi\ll1/\xi$, or equivalently $y\ll 1$.  Still, this
is enough to construct a global approximation: the region where both
approximations (\eqref{eq:78} and \eqref{eq:74}) are admissible exists
for $e^{-\frac{\lambda_{l}}{\gamma}s}\ll y\ll 1$.  This region of two
overlapping approximations allows us to compare both of them and make
sure they are compatible by picking a specific value of $\varepsilon$.

On one hand, from the asymptotic analysis for the equation
\eqref{eq:11} for large $\xi$ we get
$U_{1}(\xi)=\pi/2-h\xi^{-\gamma}+\mathcal O(\xi^{-\gamma+2})$, with a
positive coefficient $h=h(d)$, so
\begin{align*}
  f_{inn}(y,s)\approx\frac{\pi}{2}-h(y/\varepsilon(s))^{-\gamma},\qquad \epsilon(s)\ll y.
\end{align*}
On the other hand, we already mentioned that the mode $\phi_{1}$
behaves as $y^{-\gamma}$ near the origin leading to
\begin{align*}
  f_{out}(y,s)=\frac{\pi}{2}+a_{l}(0)e^{-\lambda_{l}s}\phi_{l}(y)\approx\frac{\pi}{2}+a_{l}c_{l}y^{-\gamma}e^{-\lambda_{l}s},\qquad y\ll 1,
\end{align*}
where we have used $\phi_{l}(y)=c_{l}y^{-\gamma}+\mathcal O(y^{2-\gamma})$
close to $y=0$.  Demanding that $f_{out}$ matches $f_{inn}$ up to
the leading order term, one arrives at the matching condition:
\begin{align}
  \label{eq:eps}
  \varepsilon(s)\approx
\left(\frac{-a_{l}(0)c_{l}}{h}\right)^{\frac{1}{\gamma}}e^{-\frac{\lambda_{l}}{\gamma}s},
\end{align}
giving us the innermost spatial scale of the solution, as expected
$\varepsilon(s)\to0$ as $s\to\infty$.  But, by definition,
$\varepsilon$ is related to the gradient at the origin, so we end up
with the following blow-up rate
\begin{align*}
  \frac{1}{R(t)}=\lvert\partial_{r}u(0,t)\rvert=\frac{\lvert \partial_{y}f(0,s)\rvert}{\sqrt{T-t}}=\frac{U_{1}'(0)}{\varepsilon(s)\sqrt{T-t}}\approx \frac{\kappa}{(T-t)^{-\frac{1}{2}-\frac{\lambda_{l}}{\gamma}}}.
\end{align*}
In the above formula the constant $\kappa$ depends on initial only
(through $a_{l}(0)$).

\begin{remark}
  The above computations fail for neutral eigenvalues
  ($\lambda_{l}=0$). Naturally, in that case it is necessary to
  include the interaction coming from the nonlinear term, which leads
  to logarithmic corrections to the blow-up rate.  In this paper we
  deal only with the blow-up rates coming from the strictly positive
  eigenvalues, so we don't go into detail how to derive the blow-up
  rates when $\lambda_{l}=0$, instead an interested reader is referred
  to \cite{Biernat2014}.
\end{remark}

In this paper we prove that one can find an initial data, leading to
smooth solutions to \eqref{eq:u} blowing up at time $T$ with the
blow-up rate
\begin{align}
  \label{eq:82}
  R(t)\approx (T-t)^{\frac{1}{2}+\frac{\lambda_{l}}{\gamma}},
\end{align}
where $\lambda_{l}$ stands for any positive eigenvalue of the operator
$A$.  Replacing $\lambda_{l}$ with its explicit value
$\lambda_{l}=-\frac{\gamma}{2}+l$ (see Lemma~\ref{th:extension}) and incorporating
the condition $\lambda_{l}>0$, we arrive at the blow-up rate stated
in Theorem~\ref{th:main-rate}.

\section{Preliminaries}
The initial data that we use in our construction is discontinuous,
thus we have to prove a well-posedness of the Cauchy problem
\eqref{eq:u} in the relevant space of discontinuous
functions.  Establishing such a result is only possible thanks to the
equivariant symmetry the ansatz \eqref{eq:84}, which ensures that the
singularity can only occur at the origin, and discontinuities at other
points $r>0$ will simply be smoothed out by the heat kernel.  In
particular, the Lemma below implies that the Cauchy problem
\eqref{eq:u} is well posed even for initial data with infinite initial
Dirichlet energy.  For comparison, the proof well-posedness for
general maps $F$ without any symmetry assumptions requires slightly
larger regularity of initial data:
$\lVert\nabla F_{0}\rVert_{\infty}<\infty$ (this result was first
derived by \cite{Eells1964}, see also \cite{Lin2008} or
\cite{Taylor2011} for more modern and concise approaches).  The
discontinuity of our initial data, however, does not play any
essential role in our argument; it simply removes the necessity of
introducing smoothed out indicator functions and therefore simplifies
computations.

The first part of this section is devoted to restating \eqref{eq:u} as
a proper Cauchy problem and to showing its well-posedness and some
estimates on a solution $u$ in
Proposition~\ref{th:existence-u}.  Then, we take a closer look at the
singular solution $\pi/2$ that we already used to derive a
non-rigorous blow-up rate.  We show that any solution $u$ from
Proposition~\ref{th:existence-u} can be interpreted as a classical
solution in a Hilbert space arising from studying the linear stability
of the solution $\pi/2$.

\begin{Prop}
  \label{th:existence-u}
  Assume that $K:=\sup_{r\ge0}\lvert u_0(r)/r\rvert<\infty$, then
  there is $t_{1}(u_0)>0$ such that there exists a unique solution
  $u\in C^\infty((0,t_1]\times [0,\infty))$ to \eqref{eq:6} fulfilling
  \begin{subequations}
    \begin{align}
      \label{eq:63}
      \lvert\partial_r^{(n)}\partial_t^{(k)} (u(r,t)/r)\rvert\le
      c_{n,k}(t,K)<\infty,\qquad t\in(0,t_{1}]\\
      \label{eq:64}
      \lim_{t\to0+}u(r,t) = u_0(r) \text{ a.e.}
    \end{align}
  \end{subequations}
\end{Prop}

\begin{proof}
  The linearization of \eqref{eq:u} around $u=0$,
  \begin{align*}
    \partial_{t}u=\partial_{rr}u+\frac{d-1}{r}\partial_{r}u-\frac{d-1}{r^{2}}u,
  \end{align*}
  contains a singular potential term $\frac{d-1}{2r^{2}}$.  As our
  first step we get rid of this singular term by introducing a new
  variable $\Phi(r,t)=u(r,t)/r$.  Consequently, the Cauchy problem
  \eqref{eq:u} in terms of $\Phi$ becomes
  \begin{align}
    \label{eq:Phi}
    \partial_{t}\Phi=\Delta\Phi+\mathcal F(\Phi), \qquad \Phi(r,0)=\Phi_0(r)=\frac{u_0(r)}{r},\qquad r\in \mathbb [0,\infty)
  \end{align}
  with $\Delta$ denoting the radial part of Laplacian on
  $\mathbb R^{d+2}$ and $\mathcal F(\Phi)$ standing for the nonlinear
  part
  \begin{align*}
    \mathcal F(\Phi)&=-\frac{d-1}{2r^{3}}\sin(2r\Phi)+\frac{d-1}{r^{2}}\Phi\\
                    &=4(d-1)\Phi^3\left(\frac{(2r\Phi)-\sin(2r\Phi)}{(2r\Phi)^3}\right).
  \end{align*}
  The last formula reveals that the nonlinear term behaves much better
  than it looks: it is smooth in both $r$ and $\Phi$, it
  is bounded by $\lvert\mathcal F(\Phi)\lvert\lesssim \Phi^{3}$
  independently of $r$ and positive $F(\Phi)>0$ for positive $\Phi$
  (note that we write ``$A\lesssim B$'' if $A\le C B$ with some
  inessential constant $C>0$).

  Slightly abusing the notation, let us extend the notion of $\Phi$ to
  a function on $\Phi (x,t):\mathbb R^{d+2}\times [0,t_{1})$ with
  $x\in \mathbb R^{d+2}$ such that $r=\lvert x\rvert$.  The nonlinear
  term is smooth and therefore locally Lipschitz in $L^{\infty}$ norm,
  so the general results for parabolic equations (see
  e.g. \cite{Lunardi1995} Proposition 7.3.1 or \cite{Lunardi2004}
  Section 6.3) imply that for every $\Phi_{0}\in L^{\infty}$
  (equivalently for every $u_{0}$ such that $u_{0}(r)/r$ is bounded)
  there is $t_{1}>0$ such that a unique mild solution to
  \eqref{eq:Phi} exists for $t\in(0,t_{1})$ with $\Phi$,
  $\partial_{t}\Phi$, $\nabla_{x}\Phi$, $\Delta\Phi$ continuous as
  functions on $\mathbb R^{d+2}\times(0,t_{1})$.  By a mild solution
  we understand an $L^{\infty}$ solution to the integral equation
  \begin{align}
    \label{eq:77}
    \Phi(t,\cdot)=S(t)\Phi_{0}(\cdot)+\int_{0}^{t}S(t-s)\mathcal F(\Phi(s,\cdot))\,ds,\qquad t\in(0,t_{1}),
  \end{align}
  with $S(t):L^{\infty}\to L^{\infty}$ being the heat semigroup
  generated by the Laplacian
  \begin{align*}
    S(t)\psi(x)=\int_{\mathbb R^{d+2}}G(x-y,t)\psi(y)\,dy,\qquad G(x,t)=\frac{1}{(4\pi t)^{(d+2)/2}}e^{-\frac{\lvert x\rvert^{2}}{4t}}.
  \end{align*}
  The smoothness and estimates \eqref{eq:63} follow
  directly from taking derivatives of \eqref{eq:77} and moving them
  under the integral sign followed by taking the $L^{\infty}$ norm of
  the result.

  As for \eqref{eq:64}, we combine a classical result that
  $S(t)\Phi_{0}\to \Phi_{0}$ almost everywhere as $t\to0$ with the
  following bound on the nonlinear term:
  \begin{align*}
    \left\lvert\int_{0}^{t}S(t-s)\mathcal F(\Phi(s,\cdot))\,ds\right\rvert\le\int_{0}^{t}S(t-s)\mathcal (\Phi(s,\cdot))^{3}\,ds\le t\sup_{s\in(0,t_{1})}\lVert\Phi(s,\cdot)\rVert_{\infty}^{3}.
  \end{align*}
  The latter ensures that the nonlinear term tends uniformly to zero as $t\to0$.
\end{proof}

\begin{remark}
  A classical result on second order parabolic equations
  (cf. for instance, \cite{Protter1984}) implies that the comparison
  principle holds for $\Phi$.  In consequence, a comparison principle
  holds also for $u$ and $f$. This will be of great use in
  Section~\ref{sec:proof-main-result}.
\end{remark}

As demonstrated in the previous section the linear stability analysis
of the (singular) stationary solution $f(y,s)=\frac{\pi}{2}$ can be
used to derive the blow-up rate and other quantitative properties of
blow-up solutions.  Therefore it seems natural to analyze the original
problem \eqref{eq:u} using the following similarity variables
\begin{align}
  \label{eq:7bb}
  y=\frac{r}{\sqrt{T-t}},\qquad s=-\log(T-t), \qquad \psi(s)(y)=u(r,t)-\frac{\pi}{2},
\end{align}
so that $\psi(s)=0$ corresponds to $f(y,s)=\frac{\pi}{2}$.  In
self-similar variables \eqref{eq:7bb} equation \eqref{eq:u} becomes
\begin{align}
  \label{eq:65}
  \psi'(s)&=-\mathcal A\psi(s)+F(\psi(s))\\
  F(\phi)(y)&=\frac{d-1}{2y^{2}}(\sin(2\phi(y))-(2\phi(y))),
\end{align}
with the formal operator $\mathcal A$ defined as
\begin{align}
  \label{eq:55}
  -\mathcal A\phi(y)&=\frac{1}{\rho}\frac{d}{dy}\left(\rho \frac{d\phi}{dy}(y)\right)+\frac{d-1}{y^{2}}\phi(y),\\
  \rho&=y^{d-1}e^{-\frac{y^{2}}{4}}.
\end{align}
It remains to restate the formal problem \eqref{eq:65} in a concrete
Banach space (or in this case, Hilbert space).  The Sturm-Liouville
form of $\mathcal A$ suggests to consider the Hilbert space
\begin{align*}
  \mathcal H:=L^{2}([0,\infty),\rho\,dy)=\left\{f\in L^{2}_{loc}\,|\,\int_{0}^{\infty}f(y)^{2}\rho\,dy\right\}.
\end{align*}
with the standard scalar product
\begin{align*}
  \langle f,g\rangle= \int_{0}^{\infty}f(y)g(y)\rho\,dy
\end{align*}
and a norm denoted as
$\lVert\cdot\rVert = \sqrt{\langle\cdot,\cdot\rangle}$.

As we shall see later on, the space $\mathcal H$ was chosen to fulfill
two conditions: the operator $-\mathcal A$ can be represented as a
semi-bounded symmetric operator that has a self-adjoint extension.
To prove these two results we need the following Hardy-type inequality.
\begin{Lem}
  \label{th:hardy}
  For $\phi\in C_{0}^{\infty}(\mathbb R_{+})$ we have
  \begin{align}
    \label{eq:66}
    \int_0^{\infty}\left( \phi^{\prime}(y) \right)^2 \rho dy +
    \left(\alpha^2  - \left(d - 2 \right) \alpha \right) \int_0^{\infty}\frac{\phi (y)^2}{y^2}\rho dy
    \geq -\frac{\alpha}{2} \int_0^{\infty}\phi (y)^2 \rho dy.
  \end{align}
\end{Lem}

\begin{proof}
We first notice that
  \begin{align*}
    0\le \left(\phi'+\frac{\alpha\phi}{y}\right)^{2}=(\phi')^{2}+\frac{\alpha}{r}(\phi^{2})'+\frac{\alpha^{2}\phi^{2}}{y^{2}}.
  \end{align*}
  Integration over $[0,\infty)$ with weight $\rho$ gets us to
  \begin{equation*}
    0 \leq
    \int_0^{\infty} \left( \phi^{\prime} \right)^2 \rho dy
    + \alpha \int_0^{\infty} \frac{\left( \phi^{2} \right)^{\prime}}{y} \rho dy
    + \alpha^2 \int_0^{\infty} \frac{\phi^2}{y^2} \rho dy.
  \end{equation*}
  Integration by parts applied to the middle term on the right hand
  side transforms it to
  \begin{equation*}
    \int_0^{\infty} \frac{\left( \phi^{2} \right)^{\prime}}{y} \rho dy
    = -\int_0^{\infty} \left( \frac{d-2}{y^2} - \frac{1}{2} \right) \phi^2 \rho dy.
  \end{equation*}
  Rearranging the terms we arrive at \eqref{eq:66}.
\end{proof}

The inequality \eqref{eq:66} implies that for any
$\phi\in C^{\infty}_{0}(\mathbb R_{+})$ the operator $-\mathcal A$
with $D(\mathcal A):=C^{\infty}_{0}(\mathbb R_{+})$ is bounded from
below.  To see that we rewrite $\langle\mathcal A\phi,\phi\rangle$,
using its Sturm-Liouville form \eqref{eq:55}, as
\begin{align}
  \label{lower0}
  \langle \mathcal{A} \phi, \phi \rangle
  & = \int_0^{\infty} \left( \phi^{\prime}(y) \right)^2 \rho dy - (d-1) \int_0^{\infty} \frac{\phi(y)^2}{y^2} \rho dy \notag \\
  & \geq - \left(\alpha^2  - \left(d - 2\right) \alpha + (d-1) \right) \int_0^{\infty}\frac{\phi(y)^2}{y^2} \rho dy
    - \frac{\alpha}{2} \int_0^{\infty} \phi (y)^2 \rho dy.
\end{align}
The quadratic equation
\begin{align}
  \label{eq:71}
  \alpha^{2}-(d-2)\alpha+(d-1)=0,
\end{align}
has real roots for $d\ge 4+2\sqrt{2}\approx 6.828$, so by defining
$\gamma$ as the smaller root
\begin{align}
  \label{eq:def-gamma}
  \gamma=\frac{1}{2}(d-2-\omega),\qquad \omega=\sqrt{d^{2}-8d+8}
\end{align}
and picking $\alpha=\gamma$ we get the the lower bound
\begin{align}
  \label{eq:67}
  \langle \mathcal A\psi,\psi\rangle\ge -\frac{\gamma}{2} \lVert \psi\rVert^{2}.
\end{align}
Moreover, it is easy to see that $\mathcal A$ is symmetric on
$C_{0}^{\infty}(\mathbb R_{+})$, so by a standard result on symmetric
semi-bounded operators it admits a unique Friedrichs extension to a
self-adjoint operator (see e.g. \cite{Reed2}, Theorem X.23).  It is
also routine to compute the eigenvectors and eigenvalues of this
Friedrichs extension.  We summarize these results in
Lemma~\ref{th:extension}.

\begin{Lem}
  \label{th:extension}
  For $d>4+2\sqrt{2}$ the operator $\mathcal A$ with domain
  $D(\mathcal A):=C^{\infty}_{0}(\mathbb R_{+})$ can be extended to a
  semibounded self-adjoint operator $A$ with domain $D(A)$.  The
  operator $A$ satisfies the same lower bound as $\mathcal A$ and the
  spectrum of $A$ consists of countably many simple eigenvalues.  These
  eigenvalues and eigenvectors are given by
  \begin{align}
    \label{eq:25}
    \lambda_{n}=-\frac{\gamma}{2}+n,\qquad \phi_{n}(y)=\mathcal N_{n}y^{-\gamma}L^{(\omega/2)}_{n}\left(\frac{y^{2}}{4}\right), \qquad n\ge0,
  \end{align}
  where $L^{(\alpha)}_{n}$ denotes the generalized Laguerre polynomial
  of order $n$.  The normalization constant
  \begin{align*}
    \mathcal N_{n}=\sqrt{\frac{\Gamma(1+n)}{\Gamma(1+n+\omega/2)}}
  \end{align*}
  ensures that $\lVert\phi_{n}\rVert=1$, while the asymptotic
  expansion of $L^{(\alpha)}_{n}$ yields the following behavior for
  $\phi_{n}$ at $0$ and $\infty$, respectively:
  \begin{subequations}
    \label{eq:eigen}
    \begin{align}
      \label{eq:eigen0}
      \phi_{n}(y)&=\alpha_{n}y^{-\gamma}+\mathcal O(y^{-\gamma+2}),\\
      \label{eq:eigen1}
      \phi_{n}(y)&=\beta_{n}y^{2\lambda_{l}}+\mathcal O(y^{2\lambda_{l}-2})=\beta_{n}y^{-\gamma}y^{2l}+\mathcal O(y^{2\lambda_{l}-2}).
    \end{align}
  \end{subequations}
  The coefficients $\alpha_{n}$ and $\beta_{n}$ are given \cite{nist} by
  \begin{align*}
    \alpha_{n}&=\frac{1}{\Gamma(1+\omega/2)}\sqrt{\frac{\Gamma(1+n+\omega/2)}{\Gamma(1+n)}}\approx n^{\omega/4}\\
    \beta_{n}&=4^{-n}\frac{1}{\sqrt{\Gamma(1+n)\Gamma(1+n+\omega/2)}}\approx \frac{n^{-\omega/4} 4^{-n}}{n!}
  \end{align*}
  with their behavior for large $n$ denoted by '$\approx$'.
\end{Lem}

With the Hilbert space and the self-adjoint operator
$A:D(A)\subset \mathcal H\to \mathcal H$ at hand we are ready to
rephrase the problem \eqref{eq:65} as a proper Cauchy problem in
$\mathcal H$
\begin{align}
  \label{eq:68}
  \begin{split}
    \psi'(s)&=-A\psi(s)+F(\psi(s)),\qquad s>s_{0},\\
    \qquad \psi(0)&=\psi_{0}\in \mathcal H.
  \end{split}
\end{align}
(from here on we let $s_{0}=-\log(T)$).  On one hand, the
well-posedness for Cauchy problem \eqref{eq:68} is still open and, in
fact, \eqref{eq:68} might be prone to non-unique solutions
(c.f. non-unique weak solutions found for $2<d<4+2\sqrt{2}$ in
\cite{Biernat2011} or \cite{Germain2010}).  On the other hand, we
already established a different well-posedness result in
Proposition~\ref{th:existence-u}.  In the remainder of this section we show
that a solution $u(r,t)$ from Proposition~\ref{th:existence-u}, translated
to self-similar variables as a function
\begin{align}
  \label{eq:70}
  \psi_{u}(s)(y):=
  \begin{cases}
    u(r,t)-\frac{\pi}{2}& s>s_{0}\\
    u_{0}(r)-\frac{\pi}{2} & s = s_{0},
  \end{cases}
\end{align}
can be regarded a classical solution to \eqref{eq:68} with initial
data $\psi_{0}=\psi_{u}(s_{0})$.  For the proof of this statement and
the definition of a classical solution we refer to
Lemma~\ref{th:classical} below.  Once we establish that $\psi_{u}$ is
indeed a classical solution to \eqref{eq:68} we immediately get the
following corollary (see for example \cite{Lunardi1995} or
\cite{Lunardi2004}).

\begin{Cor}
  \label{th:duhamel}
  Given initial data $\psi_{0}(y)=u_{0}(r)-\frac{\pi}{2}$ with
  $\lvert u_{0}(r)/r\rvert$ bounded, the Cauchy problem \eqref{eq:68}
  has a unique classical solution $\psi_{u}$ defined in
  \eqref{eq:70}.  The classical solution $\psi_{u}$ solves the
  integral equation
  \begin{align}
    \label{eq:69}
    \psi_{u}(s)=e^{-(s-s_{0})A}\psi_{u}(s_{0})+\int_{s_{0}}^{s}e^{-(s-\tau)A}F(\psi_{u}(\tau))\,d\tau,\qquad s\ge s_{0}
  \end{align}
  where $e^{-sA}:\mathcal H\to \mathcal H$ is the
  $C_{0}$-semigroup generated by $-A$ and given explicitly as
  \begin{align*}
    e^{-sA}v=\sum_{n=0}^{\infty}e^{-\lambda_{n}s}\langle v,\phi_{n}\rangle\phi_{n},\qquad s>0
  \end{align*}
  for any $v\in \mathcal H$.
\end{Cor}

So, although we do not prove the well-posedness of \eqref{eq:68} we
can still employ the formula \eqref{eq:69}. There are several benefits
that we gain by reformulating \eqref{eq:u} as \eqref{eq:69}.  First of
all, we gain the notion of continuity in $\mathcal H$ starting from
$s=s_{0}$, which corresponds to $t=0$.  This is an improvement
compared to $u$ being continuous in $L^{\infty}$ only for $t>0$, in
other words we have $u\in C((0,t_{1}],L^{\infty})$ but
$\psi_{u}\in C([s_{0},s_{1}],\mathcal H)$.  Secondly, we are able to
explicitly use eigenvalues of $A$ in a priori estimates on
$\psi_{u}$.  And finally, the heat kernel associated to the semigroup
$e^{-sA}$ is known in an explicit form and its action can be bounded
by maximal functions (cf. Lemma~\ref{estim-semigroup} further on).
We devote the rest of this section to proving that \eqref{eq:70} is,
in fact, a classical solution to \eqref{eq:68}.

\begin{Lem}
  \label{th:classical}
  The function $\psi_{u}$ defined in \eqref{eq:68} is a classical solution to
  \eqref{eq:68}, that is
  \begin{align}
    \label{eq:72}
    \psi_{u}\in C([s_{0},s_{1}],\mathcal H)\cap C((s_{0},s_{1}],D(A))\cap
    C^{1}((s_{0},s_{1}],\mathcal H)
  \end{align}
  (with $s_{1}=-\log(T-t_{1})$) and
  \begin{align}
    \label{eq:75}
    \left(s\to F(\psi_{u}(s))\right)\in C([s_{0},s_{1}],\mathcal H)\cap L^{1}([s_{0},s_{1}],\mathcal H)
  \end{align}
  and finally $\psi_{u}$ solves
  $\psi_{u}'(s)=-A\psi_{u}(s)+F(\psi_{u}(s))$ for $s>s_{0}$.
\end{Lem}

\begin{proof}
  We already showed that $u$ is smooth and solves \eqref{eq:u} thus,
  by construction, $\psi_{u}$ solves
  $\psi_{u}'(s)=-A\psi_{u}(s)+F(\psi_{u}(s))$ for $s>s_{0}$, so it
  remains to confirm the appropriate continuity conditions
  \eqref{eq:72} and \eqref{eq:75}.

  To this end, we note that \eqref{eq:63} from
  Proposition~\ref{th:existence-u} implies that for any fixed $t>0$ all
  derivatives of $u$ are bounded near the origin and have at most
  linear growth for large $r$
  \begin{align*}
    \lvert\partial_{t}^{(n)}\partial_{r}^{(k)}u(r,t)\rvert\lesssim 1+r,\qquad t>0.
  \end{align*}
  These bounds carry over to bounds on derivatives of $\psi_{u}$ for fixed $s>s_{0}$
  (possibly with different $\beta_{n,k}$)
  \begin{align*}
    \lvert\psi_{u}(s)(y)\rvert&\lesssim 1+y,\qquad&\lvert\partial_{y}\psi_{u}(s)(y)\rvert&\lesssim 1+y,\\
    \lvert\partial_{yy}\psi_{u}(s)(y)\rvert&\lesssim 1+y,\qquad&\lvert\partial_{s}\psi_{u}(s)(y)\rvert&\lesssim 1+y^{2},
  \end{align*}
  (the additional power for the time derivative $\partial_{s}$ comes
  from the last term in
  $\partial_{s}=e^{-s}\partial_{t}+\frac{1}{2}r\partial_{r}$).  Now,
  the action of $A$ on $\psi_{u}(s)$ can be bounded by
  \begin{align*}
    \lvert A\psi_{u}(s)(y)\rvert&=\left\lvert\partial_{yy}\psi(s)(y)+\left(\frac{d-1}{y}-\frac{y}{2}\right)\partial_{y}\psi_{u}(s)(y)+\frac{d-1}{y^{2}}\psi_{u}(s)(y)\right\rvert\\
    &\lesssim y^{-2}+y^{2}.
  \end{align*}
  From the above we deduce that the norm
  \begin{align*}
    \lVert A\psi_{u}(s)\rVert^{2}\lesssim\int_{0}^{\infty}\left(y^{-2}+y^{2}\right)^{2}y^{d-1}e^{-\frac{y^{2}}{4}}\,dy
  \end{align*}
  is finite, thus $\psi_{u}(s)\in D(A)\subset \mathcal H$ for $s>s_{0}$.  In a
  completely analogous fashion one can show that
  $\psi_{u}'(s),F(\psi_{u}(s))\in \mathcal H$ for $s>s_{0}$.
  At this point the continuity for $s>s_{0}$ follows from smoothness
  of $\psi_{u}(s)(y)$ in $s$ and $y$ for $s>s_{0}$ and the dominated
  convergence theorem.

  We therefore established that \eqref{eq:72} and \eqref{eq:75} hold
  on the interval $(s_{0},s_{1}]$.  To finalize the proof we note that
  by \eqref{eq:64} $u(r,t)\to u_{0}(r)$ for $t\to0$ almost everywhere,
  thus $\psi_{u}(s)(y)\to\psi_{u}(s_{0})$ almost everywhere when
  $s\to s_{0}$.  So it suffices to show that $\psi_{u}(s_{0})$ and
  $F(\psi_{u}(s_{0}))$ are in $\mathcal H$ and apply the dominated
  convergence theorem.  Indeed, for $\lvert u_{0}(r)/r\rvert$ bounded
  we have $\lvert\psi_{u}(s_{0})\rvert\lesssim 1+y$, which implies
  $\psi_{u}(s_{0}),F(\psi_{u}(s_{0}))\in\mathcal H$ and the continuity
  at $s_{0}$ follows.
\end{proof}

\section{Proof of the main result}
\label{sec:proof-main-result}
In this section we construct a solution to \eqref{eq:u} that blows up
at time $t=T$, or equivalently a solution to \eqref{eq:68} defined for
all times $s\ge s_{0}$.  This construction is based on the matched
asymptotics method used to construct a formal solution and requires
the approximations derived in
Section~\ref{sec:derivation-blow-up}.  The basic idea is to take
initial data that is already close to the formal solution and fine
tune it using a finite number of parameters so that it stays close to
the anticipated formal solution for as long as we like.  We use the
initial time $s_{0}$ as a ``bettering'' parameter, that is by taking
$s_{0}$ large enough our initial data gets closer to the formal
solution taken at time $s_{0}$; let us also remind that
$s_{0}=-\log(T)$, so taking $s_{0}$ large corresponds to a small
blow-up time $T$.

First, let us remind that the inner part of the approximation was
based on the solution $U_{1}$ to the ordinary differential equation
\begin{align}
  \label{eq:73}
  U''(r)+\frac{d-1}{r}U'(r)-\frac{d-1}{2r^{2}}\sin(2U)=0,\qquad U(0)=0
\end{align}
with initial condition $U_{1}'(0)=1$; naturally $u(r,t)=U_{\alpha}(r)$
is a stationary solution to \eqref{eq:u}.  Thanks to the scaling
symmetry $r\to\lambda r$ of \eqref{eq:73}, $U_{1}$ gives rise to
solutions $U_{\alpha}$ for all $\alpha>0$ by simply defining
$U_{\alpha}(r):=U_{1}(\alpha r)$.  Lemma~\ref{th:stationary} below
establishes the basic properties of $U_{\alpha}$.  We do not prove it
here but the proof can be found in the Appendix of \cite{Biernat2014}
and it consists of the phase portrait analysis of the autonomous
equation $v''(x)+(d-2)v'(x)+(d-2)\sin(v(x))=0$ arising after changing
variables to $x=\log(r)$ and $v(x)=U(r)$.

\begin{Lem}
  \label{th:stationary}
  For $d>4+2\sqrt{2}$, there exists a family of solutions
  $U_\alpha(r)=U_{1}(\alpha r)$, (with $\alpha > 0$) to equation
  \eqref{eq:73} such that
  \begin{align*}
    U_\alpha(0)=0,\quad U_\alpha'(0)=\alpha.
  \end{align*}
  Moreover $U_{\alpha}$ is monotone
  \begin{align}
    \label{eq:U-monotone}
    U_{\alpha}'(r)>0,\qquad r>0
  \end{align}
  and its asymptotic behavior for large $r$ is
  \begin{align}
    \label{G4}
    U_\alpha(r)=\frac{\pi}{2}-h\,\alpha^{\gamma}r^{-\gamma}+\mathcal O(r^{-\gamma-2}),
  \end{align}
  where $h$ is a strictly positive constant depending only on $d$ and
  $\gamma$ is the constant given by \eqref{eq:def-gamma}.
\end{Lem}

Now let us fix $\alpha$ to be
\begin{align*}
  \alpha:=\left(\frac{\alpha_{l}}{h}\right)^{\frac{1}{\gamma}}
\end{align*}
and define the following constants
\begin{align*}
  0<\widetilde k<k<1,\qquad 0<\widetilde\sigma<\sigma<\frac{1}{2},\\
  \omega_{l}:=\frac{\lambda_{l}}{\gamma},\qquad K:=e^{k\omega_{l}s_{0}},\qquad \widetilde K:=e^{\widetilde k\omega_{l}s_{0}}
\end{align*}
(note that $K\gg \widetilde K\gg 1$ for $s_{0}\gg 1$).  The constants
$k$, $\sigma$ and $\widetilde\sigma$ will be further restricted in
Lemma~\ref{initial-data-long-time},
Lemma~\ref{lem:estimate-short-inhomogeneous} and
Lemma~\ref{lem:estimate-short-inhomogeneous-extension} but the upshot
is that the constant $k$ should be taken close to $1$, while $\sigma$
and $\widetilde\sigma$ should be chosen close to $0$.  We can now take
$q=(q_{0},q_{2},\dots,q_{l-1})\in\mathbb R^{l}$ and a function
$\psi_{0,q}$, which will serve as the initial data, such that
\begin{subequations}
  \label{eq:20}
  \begin{align}
    \label{eq:14}
    \psi_{0,q}(y)&=U_{\alpha}(ye^{\omega_{l}s_{0}})-\frac{\pi}{2}&y&\in [0,\widetilde K e^{-\omega_{l}s_{0}})\\
    \psi_{0,q}(y)&=\sum_{n=0}^{l-1}q_{n}\phi_{n}(y)-e^{-\lambda_{l}s_{0}}\phi_{l}(y) & y&\in [\widetilde K e^{-\omega_{l}s_{0}},e^{\widetilde\sigma s_{0}})\\
    \label{eq:21}
    \lvert \psi_{0,q}(y)\rvert&\le \frac{\pi}{2}& y&\in [e^{\widetilde\sigma s_{0}},\infty).
  \end{align}
\end{subequations}
One can immediately notice that with $q=0$ the second part of the
definition \eqref{eq:20} corresponds to the formal approximation we
made in \eqref{eq:74} with $a_{l}(0)=-1$, while the first part is just
\eqref{eq:78} with $\varepsilon(s)$ coming from \eqref{eq:eps} (the
two parts meet divided at $y=\widetilde Ke^{-\omega_{l}s}\ll 1$).  The
parameters $q$ in \eqref{eq:20} stand in front of the modes of
lower-order than $\phi_{l}$.  The presence of the third term
\eqref{eq:21} is due to the polynomial growth of $\phi_{n}$ at
infinity, which we have to cut off at some point to produce a bounded
solution, that is $u(r,t)$ enclosed in a strip $0\le u(r,t)\le \pi$.

As a matter of fact, function $\psi_{0,q}$ is discontinuous but it is easy to see that the
associated
$u_{0,q}(r):=\psi_{0,q}(y)+\frac{\pi}{2}=\psi_{0,q}(e^{-s_{0}/2}r)+\frac{\pi}{2}$
fulfills the assumptions of Corollary~\ref{th:duhamel}, thus a
solution to \eqref{eq:68} with initial data $\psi_{0,q}$ exists for
some short time $s_{1}>s_{0}$.  From now on, we shall refer to a
solution to \eqref{eq:68} with initial data $\psi_{0,q}$ as
$\psi_{q}$.

Let us now define a property of solution $\psi_{q}$, which serves as
the basis of our topological argument.  We say that $\psi_{q}$ has the
property $\mathcal W^{\theta}_{ s_{0}, s_{1}}$ (or
$\psi_{q}\in \mathcal W^{\theta}_{ s_{0}, s_{1}}$ for short) if
the following estimate holds for $ s_{0}\le  s\le  s_{1}$
\begin{align}
  \label{eq:58}
  \lvert \psi_{q}(s)(y)+e^{-\lambda_{l} s}\phi_{l}(y)\rvert&< \theta \eta e^{-\lambda_{l} s}(y^{-\gamma}+y^{2\lambda_{l}})&& \text{ for } Ke^{-\omega_{l} s} \le y\le e^{\sigma s}
\end{align}
The parameter $\eta\in(0,\alpha_{l})$ should be treated as a fixed
number, chosen for the remainder of this paper according to
Lemma~\ref{lem:intermediate-implications} below.
Basing on the definition of $\mathcal W^{\theta}_{ s_{0}, s_{1}}$ we
define the following subset of $\mathbb R^{l}$
\begin{align}
  \label{eq:56}
  \mathcal U_{ s_{0}, s_{1}}:=\left\{q\in \mathbb R^{l}\,:\, \psi_{q}\in \mathcal W^{1}_{ s_{0}, s_{1}}\right\}\cap B_{e^{-\lambda_{l} s_{0}}}(0).
\end{align}
It turns out that having $q\in\mathcal U_{ s_{0}, s_{1}}$ provides
sufficient information to ensure that the solution $\psi_{q}$ is close to
a re-scaled stationary solution near the origin and stays bounded in
the external region $y>e^{\sigma s}$ in the sense of the following
Lemma.

\begin{Lem}
  \label{lem:intermediate-implications}
  Take $\delta\in(0,1)$ and $\eta=\eta(\delta)>0$ small enough.  Then
  for any $q\in\mathcal U_{ s_{0}, s_{1}}$ and
  $ s_{0}\le  s\le s_{1}$ we have
  \begin{subequations}
    \begin{align}
      \label{hyp-inn}
      U_{\alpha\delta}(ye^{\omega_{l} s})-\frac{\pi}{2}< &\psi_{q}(s)(y)<U_{\alpha/\delta}(ye^{\omega_{l} s})-\frac{\pi}{2}&y&\in[0,Ke^{-\omega_{l} s})\\
      \label{hyp-inter}
      \lvert \psi_{q}(s)(y)+e^{-\lambda_{l} s}\phi_{l}(y)\rvert &< \eta e^{-\lambda_{l} s}(y^{-\gamma}+y^{2\lambda_{l}})& y&\in[Ke^{-\omega_{l} s},e^{\sigma s})\\
      \label{hyp-out}
      \lvert \psi_{q}(s)(y)\rvert&\le\frac{\pi}{2} &y&\in[e^{\sigma s,\infty}),
    \end{align}
  \end{subequations}
  provided that $ s_{0}$ is large enough.
\end{Lem}

\begin{proof}
  Estimate \eqref{hyp-inter} follows directly from the definition of
  $\mathcal W^{1}_{ s_{0}, s_{1}}$.  According to
  Lemma~\ref{lem:inner} below, if we pick
  $\eta<\alpha_{l}(\delta^{-\gamma}-1)$ and $ s_{0}$ large enough then
  \eqref{hyp-inn} must hold.  It remains to show \eqref{hyp-out} but
  this is a straight forward implication of maximum
  principle.  Namely, from \eqref{eq:21} it follows that
  \begin{align*}
    \lvert\psi_{q}(s_{0})(y)\rvert\le \frac{\pi}{2},
  \end{align*}
  while at the boundary $y=e^{\sigma s}$ of the outer region we have
  \begin{align*}
    \lvert\psi_{q}(s)\rvert_{y=e^{\sigma s}}\sim e^{-\lambda_{l} s}(e^{\sigma s})^{2\lambda_{l}}=e^{-(1-2\sigma)\lambda_{l} s}\le e^{-(1-2\sigma)\lambda_{l} s_{0}}\ll1
  \end{align*}
  as long as $s_{0}$ is large enough.  If we now notice that
  $\psi(s)(y)=\pm\frac{\pi}{2}$ are exact solutions to \eqref{eq:68},
the maximum principle implies that the solution $\psi_{q}$ is
  confined to the strip $\lvert\psi_{q}(s)(y)\rvert\le \frac{\pi}{2}$.
\end{proof}

Let us define a map
\begin{align}
  \label{eq:15}
  P_{ s_{0}, s_{1}}(q):=\left(\langle \psi_{q}(\cdot, s_{1}),\phi_{0}\rangle,\dots,\langle \psi_{q}(\cdot, s_{1}),\phi_{l-1}\rangle\right)
\end{align}
which is analytic as a map
$P_{ s_{0}, s_{1}}:\mathcal U_{ s_{0}, s_{1}}\to\mathbb R^{l}$
(via analytic dependence on initial data).  The set
$\mathcal U_{ s_{0}, s_{1}}$ was defined in such a way that the
roots of $P_{ s_{0}, s_{1}}$ can never cross the boundary of
$\mathcal U_{ s_{0}, s_{1}}$.

\begin{Lem}
  \label{lem:interior}
  If $ s_{0}$ is large enough and
  $q\in\mathcal U_{ s_{0}, s_{1}}$ is a root of
  $P_{ s_{0}, s_{1}}$ (i.e. $P_{ s_{0}, s_{1}}(q)=0$) for some $s_1 > s_0$, then
  $q\notin\partial \mathcal U_{ s_{0}, s_{1}}$.
\end{Lem}
\begin{proof}
  Sections \ref{sec:short-time} and \ref{sec:long-time} of this paper
  consist of a series of Lemmas that altogether guarantee that if
  $P_{ s_{0}, s_{1}}(q)=0$ then for any $\nu\in (0,1)$ we can
  choose $ s_{0}\gg 1$ large enough so that
  \begin{align}
    \label{eq:38}
    \lvert\psi_{q}(s)(y)+e^{-\lambda_{l}s}\phi_{l}(y)\rvert \le \nu e^{-\lambda_{l} s}(y^{-\gamma}+y^{2\lambda_{l}})\qquad  Ke^{-\omega_{l} s}\le y\le e^{\sigma s}.
  \end{align}
  In addition to that result Lemma~\ref{lem:ball} shows that any root
  of $P_{ s_{0}, s_{1}}$ lies close to the origin in a sense that
  $q\in B_{\varepsilon e^{-\lambda_{l} s_{0}}}(0)$ for some
  $\varepsilon\ll 1$.
\end{proof}

\bigskip

\noindent
\textbf{The proof of Theorem 1.1}.
The above Lemma is crucial in applying the following topological argument
.
Lemma~\ref{lem:interior} and the general homotopy
principle guarantee that the degree of zero is conserved:
\begin{align*}
  \text{deg}(P_{ s_{0}, s_{1}},\mathcal U_{ s_{0}, s_{1}},0)=\text{deg}(P_{ s_{0}, s_{0}},\mathcal U_{ s_{0}, s_{0}},0)
\end{align*}
as long as $\mathcal U_{ s_{0}, s}\ne\emptyset$ for any
$ s_{0}\le s\le s_{1}$.  A direct computation uncovers that
$P_{ s_{0}, s_{0}}$ is a small perturbation of identity, provided that
$ s_{0}\gg1$, so
\begin{align*}
  \text{deg}(P_{ s_{0}, s_{1}},\mathcal U_{ s_{0}, s_{1}},0)=\text{deg}(P_{ s_{0}, s_{0}},\mathcal U_{ s_{0}, s_{0}},0)=1.
\end{align*}
This leads to a conclusion that as long as
$\mathcal U_{ s_{0}, s_{1}}\ne\emptyset$ it must contain a root of
$P_{ s_{0}, s_{1}}$.

By continuous dependence of solutions on initial data, we also claim
that $\mathcal U_{ s_{0}, s_{1}}\ne \emptyset$ if $ s_{0}$ is
sufficiently large.  Assume now that maximal time of existence of the
set $\mathcal U_{ s_{0}, s}$ is finite, i.e.
\begin{align*}
   s^{*}=\sup\{ s> s_{0}\,:\,\mathcal U_{ s_{0}, s}\ne\emptyset\}<\infty.
\end{align*}
The set $U_{ s_{0}, s^{*}}$ is nonempty and thus it contains a root
$P_{ s_{0}, s^{*}}$ but for any such root (say we call it $q^{*}$) we
must have $\psi_{q^{*}}\in\mathcal W^{\theta}_{ s_{0}, s^{*}}$ for
$0<\theta<1$.  But from smoothness of $\psi_{q^{*}}$, we deduce that
$\psi_{q^{*}}\in\mathcal W^{1}_{ s_{0}, s^{*}+\eta}$ for some
$\eta>0$, which contradicts $ s^{*}<\infty$.

At this point, $\mathcal U_{ s_{0},\infty}\ne\emptyset$ along with
Lemma~\ref{lem:intermediate-implications}, imply our main theorem
\ref{th:main-rate}.  Indeed, from \eqref{hyp-inn} and
$ye^{s\omega_{l}}=r(T-t)^{l/\gamma}$ we have
\begin{align*}
  U_{\alpha\delta}(r/(T-t)^{l/\gamma})< u(r,t) < U_{\alpha/\delta}(r/(T-t)^{l/\gamma})
\end{align*}
for $r$ sufficiently close to the origin.  If we now divide by $r$ and take
the limit $r\to0$ (mind that $u(0,t) = 0$ and $u(r,t)$ is smooth for $t>0$) and apply
$U_{\alpha}'(0)=\alpha$ we get
\begin{align*}
  {\alpha\delta}(T-t)^{-l/\gamma} \le \partial_{r}u(0,t) \le \alpha/\delta (T-t)^{-l/\gamma}.
\end{align*}
This last estimate yields
\begin{align*}
  \frac{1}{\partial_{r}u(0,t)}\propto (T-t)^{l/\gamma},
\end{align*}
which is just the blow-up rate that we claimed.

We complete this section with the proof of a priori estimates for the
inner layer.

\begin{Lem}
  \label{lem:inner}
  For any $0<\delta<1$ and $\eta<\alpha_{l}(\delta^{-\gamma}-1)$ the bound
  \begin{align}
    \label{eq:24}
    \lvert\psi_{q}(s)(y)+e^{-\lambda_{l} s}\phi_{l}(y)\rvert<\eta e^{-\lambda_{l} s}(y^{-\gamma}+y^{2\lambda_{l}}),\qquad Ke^{\omega_{l} s}\le y\le e^{\sigma s}
  \end{align}
  for $s_{0}\le s\le s_{1}$ implies
  \begin{subequations}
    \begin{align}
      \label{eq:22}
      U_{\alpha\delta}(ye^{\omega_{l} s})-\frac{\pi}{2}< \psi_{q}(s)(y)< U_{\alpha/\delta}(ye^{\omega_{l} s})-\frac{\pi}{2},\qquad 0\le y\le Ke^{-\omega_{l} s}
    \end{align}
  \end{subequations}
  for $s_{0}\le s\le s_{1}$, provided that $s_{0}$ is taken
  sufficiently large.
\end{Lem}

\begin{proof}
  In this proof, we always assume that $0\le\xi\le K$, we also remind
  that $K~=~e^{k\omega_{l}s_{0}}$ can be made arbitrarily large by
  taking large $s_{0}$.  Rewriting the equation \eqref{eq:68} in inner
  variables ($\xi:=ye^{\omega_{l}s}$ and
  $\Phi(\xi, s)=\psi_{q}(s)(y)+\frac{\pi}{2}$) leads to
  \begin{align*}
    0=\mathcal U(\Phi):=\partial_{ s}\Phi+\left(\frac{1}{2}+\omega_{l}\right)\xi\partial_{\xi}\Phi-e^{2\omega_{l s}}\left(\partial_{\xi\xi}\Phi+\frac{d-1}{\xi}\partial_{\xi}\Phi-\frac{d-1}{2\xi^{2}}\sin(2\Psi)\right).
  \end{align*}
  The initial data $\psi_{0,q}$ translate to
  \begin{align}
    \label{eq:39}
    \Phi(\xi, s_{0})=U_{\alpha}(\xi),
  \end{align}
  while the bounds \eqref{eq:24}, along with
  $\phi_{l}(y)\approx\alpha_{l}y^{-\gamma}$ near the origin, lead to the
  following estimate at the boundary, $\xi=K$,
  \begin{align}
    \label{eq:46}
    \frac{\pi}{2}-(\alpha_{l}+\eta)K^{-\gamma}\le \Phi(K, s)\le \frac{\pi}{2}-(\alpha_{l}-\eta) K^{-\gamma}
  \end{align}
  up to the leading order in $K$ (mind that by taking $ s_{0}$ large
  enough we can make these higher order terms arbitrarily small),
  where we have used the asymptotics of eigenfunctions
  \eqref{eq:eigen0}.

  Substituting $\overline\Phi(\xi, s):=U_{\alpha/\delta}(\xi)$ leads to
  \begin{align}
    \label{eq:27}
    \mathcal U(\overline\Phi)=\left(\frac{1}{2}+\omega_{l}\right)\xi U_{\alpha/\delta}'(\xi)\ge0,
  \end{align}
  which makes $\overline\Phi(\xi, s)$ a natural candidate for a
  supersolution.  It remains to verify that it fulfills the respective
  inequalities involving initial and boundary conditions in the region
  $0\le\xi\le K$.  The monotonicity of $U_{\alpha}$ (see inequality
  \eqref{eq:U-monotone}) immediately leads to
  \begin{align*}
    \Phi(\xi, s_{0})=U_{\alpha}(\xi)\le U_{\alpha/\delta}(\xi)=\overline\Phi(\xi),
  \end{align*}
  On the other hand, the asymptotics of $U_{\alpha/\delta}(\xi)$ as
  $\xi\to\infty$ yields
  \begin{align*}
    \overline\Phi(K, s)=\frac{\pi}{2}-h(\alpha/\delta)^{\gamma} K^{-\gamma}.
  \end{align*}
  Comparing this with \eqref{eq:46}, we get
  $\overline\Phi(K, s)>\Phi(K, s)$ if
  \begin{align}
    \label{eq:61}
    0<\eta<\alpha_{l}\left(1-\delta^{\gamma}\right),
  \end{align}
  which makes $\overline\Phi(\xi, s)$ a supersolution.  In terms of
  $\psi_{q}$ this means that
  \begin{align*}
    \psi_{q}(s)(y)\le \overline\Phi(\xi, s)-\frac{\pi}{2}=U_{\alpha/\delta}(ye^{\omega_{l} s})-\frac{\pi}{2},
  \end{align*}
  which proves the upper part of the bound \eqref{eq:22}.

  The subsolution requires a subtler approach: on one hand we need
  something we can compare to $U_{\alpha\delta}$, on the other hand we
  already showed that taking $U_{\alpha}(\xi)$ alone can only lead to
  a supersolution.  Surprisingly, a small perturbation of
  $U_{\alpha\delta}$ can serve as a subsolution; let us define
  \begin{align}
    \label{eq:47}
    \underline\Phi(\xi, s):=U_{\alpha\delta}(\xi)+e^{-2\omega_{l} s}q(\xi),
  \end{align}
  where the shape of the profile $q$ is yet to be determined.  In the
  following paragraphs, we show that $q$ can be chosen in such a way
  that $\underline\Phi$ is actually a subsolution and
  $q(\xi)\ge0$.  Then, by definition of $\underline\Phi$, the
  following series of inequalities holds:
  \begin{align*}
    U_{\alpha\delta}(\xi)\le U_{\alpha\delta}(\xi)+e^{-2\omega_{l} s}q(\xi)=\underline\Phi(\xi, s)\le \Phi(\xi, s)
  \end{align*}
  thus providing the lower bound of \eqref{eq:24}.

  Regrouping the terms leads us to
  \begin{subequations}
    \label{eq:50}
    \begin{align}
      \mathcal U(\underline\Phi)=&-\left(q''+\frac{d-1}{\xi}q'-\frac{d-1}{\xi^{2}}\cos(2U_{\alpha\delta})q\right)\\
                                 &+\left(\frac{1}{2}+\omega_{l}\right)\xi U_{\alpha\delta}'\\
      \label{eq:53}
                                 &+e^{-2\omega_{l} s}\left(-2\omega_{l}q+\left(\frac{1}{2}+\omega_{l}\right)\xi q'\right)\\
      \label{eq:54}
                                 &+e^{2\omega_{l} s}\frac{d-1}{2\xi^{2}}(\sin(2U_{\alpha\delta})-\sin(2U_{\alpha\delta}+e^{-2\omega_{l} s}q)+2\cos(2U_{\alpha\delta})e^{-2\omega_{l} s}q),
    \end{align}
  \end{subequations}
  which, despite looking complicated, has some desired
  qualities.  Namely, the last two terms can be regarded as small for
  large times (the last term is a nonlinearity of second order in
  $q$), while the first two terms are of comparable order in
  time.  This formal analysis leads to the particular choice of $q$
  that compensates for the positivity of the second term in
  \eqref{eq:50}; take $q$ solving
  \begin{align}
    \label{eq:36}
    \begin{split}
      q''+\frac{d-1}{\xi}q'-\frac{d-1}{\xi^{2}}\cos(2U_{\alpha\delta})q=\left(\beta+\frac{1}{2}+\omega_{l}\right) \xi U_{\alpha\delta}',\\
      \qquad q(0)=0, \qquad q'(0)=0.
    \end{split}
  \end{align}
  The whole expression simplifies \eqref{eq:50} significantly to
  \begin{align}
    \label{eq:37}
    \mathcal U(\underline\Phi)=-\beta\xi U_{\alpha\delta}'(\xi)+\mathcal O(e^{-2\omega_{l} s}\widetilde q(\xi))
  \end{align}
  for some $\widetilde q$ and any $\beta>0$. This, in turn, leads to
  $\mathcal U(\underline\Phi)\le 0$ (here we anticipate that the
  remainder term $e^{-2\omega_{l} s}\widetilde q(\xi)$ is small for
  our choice of $q$, we prove this below).  We would like to point
  out, that when $\beta=0$ the $\underline\Phi$ is just a second order
  approximation to a solution of $\mathcal U(\Phi)=0$ (the first order
  approximation being $U_{\alpha\delta}$).

  For large $\xi$, any solution to \eqref{eq:36} behaves like
  $q(\xi)\sim\xi^{2-\gamma}$, which means that, depending on the sign
  of $2-\gamma$, $q$ is either bounded or increases.  The first case
  leads to smallness (the whole perturbation decays exponentially in
  time), in the second case $q$ increases with $\xi$ and thus is
  maximal at the boundary of the inner region $\xi=K$.  At the
  boundary we have
  \begin{align*}
    \lvert e^{-2\omega_{l} s}q(\xi)\rvert\le \lvert e^{-2\omega_{l} s}q(K)\rvert\sim
    (Ke^{-\omega_{l} s})^{2}K^{-\gamma}
  \end{align*}
  and the right hand side is small when $ s_{0}$ is taken
  large.  The arbitrary smallness of the perturbation is necessary for
  the following comparison between initial data to hold
  \begin{align}
    \label{eq:51}
    \underline\Phi(\xi, s_{0})=U_{\alpha\delta}(\xi)+e^{-2\omega_{l} s}q(\xi)\le U_{\alpha}(\xi)=\Phi(\xi, s_{0}).
  \end{align}
  Moreover, at the boundary $\xi=K$ we have
  \begin{align*}
    \underline\Phi(K, s)=-h(\alpha\delta K)^{-\gamma}(1+\mathcal O((e^{-\omega_{l} s}K)^{2}))
  \end{align*}
  which, together with \eqref{eq:46} yields
  \begin{align}
    \label{eq:52}
    \underline\Phi(K, s)\le \Phi(K, s)
  \end{align}
  if we pick
  \begin{align}
    \label{eq:59}
    \eta<\alpha_{l}\left(\delta^{-\gamma}-1\right)
  \end{align}
  and $ s_{0}$ accordingly large.  As for the remainder terms, the
  term \eqref{eq:53} must obey the same type of bounds as $q$, while
  the last term \eqref{eq:54} is a nonlinear term of second order in
  $e^{-2\omega_{l} s}q$ thus, up to first order, it is well
  approximated by
  \begin{align*}
    \eqref{eq:54}\sim e^{-2\omega_{l} s}\widetilde q(\xi)
  \end{align*}
  where $\widetilde q(\xi)=\xi^{3}$ near $\xi=0$ and
  $\widetilde q(\xi)=\xi^{4-3\gamma}$ when $\xi$ is large.  Exploiting
  the fact that $\gamma>1$ we get
  $\xi^{4-3\gamma}=\xi^{2-\gamma}\le \xi^{2-\gamma}K^{-2(\gamma-1)}\ll
  \xi^{2-\gamma}$,
  thus the nonlinear term is certainly of smaller order of magnitude
  than $q$.  Combining \eqref{eq:37}, \eqref{eq:51} and \eqref{eq:52}
  we conclude that $\underline\Phi$ is a subsolution,
whence the comparison principle yields
  \begin{align*}
    U_{\alpha\delta}(\xi)+e^{-2\omega_{l}s}q(\xi)\le \Phi(\xi,s)
  \end{align*}
  but we are still missing the positivity of $q$ to get
  $U_{\alpha\delta}(\xi)\le \Phi(\xi,s)$ and finalize the proof.

  The positivity is easily seen if  we rewrite \eqref{eq:36} as
  \begin{align}
    \label{eq:26}
    \frac{d}{d\xi}\left(\frac{q(\xi)}{\xi U_{\alpha\delta}'(\xi)}\right)=\frac{1}{\xi^{d+1}}\int_{0}^{\xi}s^{d+1}U_{\alpha\delta}'(s)^{2}f(s)\,ds
  \end{align}
  where
  \begin{align}
    \label{eq:48}
    f(\xi)=\left(\beta+\frac{1}{2}+\omega_{l}\right) \xi U_{\alpha\delta}'(\xi)\ge0.
  \end{align}
  In the expression \eqref{eq:26} we got rid of the nonlinear term by
  turning to equation for $U_{\alpha\delta}$.  The idea for such
  treatment comes from the factorization of the operatorial form of
  the left hand side of \eqref{eq:36} and from the equation for
  $U_{\alpha\delta}$.  Nonnegativity of the source term
  \eqref{eq:48} gives $(q(\xi)/\xi U_{\alpha\delta}(\xi))'\ge0$ or
  \begin{align*}
    q(\xi)\ge \xi U_{\alpha\delta}'(\xi)\cdot\lim_{s\to0+}\frac{q(s)}{sU_{\alpha\delta}'(s)}=0,
  \end{align*}
  by the virtue of boundary conditions $q(0)=0$ and $q'(0)=0$.

  Combining \eqref{eq:61} and \eqref{eq:59} it is sufficient to pick
  $\eta$ such that
  \begin{align*}
    0<\eta<\alpha_{l}\min(\delta^{-\gamma}-1,1-\delta^{\gamma})=\alpha_{l}(\delta^{-\gamma}-1).
  \end{align*}
The proof is now complete.
\end{proof}

\section{Estimates for spectral coefficients}
\label{sec:prel-estim}

The goal of this and the remaining sections is to produce the missing
a priori estimates that we were eager enough to use in
Lemma~\ref{lem:interior} above.  That is, we aim to show that given a
solution $\psi_{q}$, such
that
\begin{align}
  \label{eq:62}
  \lvert\psi_{q}(s)(y)+e^{-\lambda_{l}s}\phi_{l}\rvert < \eta e^{-\lambda_{l}s}(y^{-\gamma}+y^{2\lambda_{l}})\qquad \text{for }Ke^{-\omega_{l}s}\le y\le e^{\sigma s}
\end{align}
(or $\psi_{q}\in\mathcal W_{s_{0},s_{1}}^{1}$ equivalently), with $q$
being the root of $P_{s_{0},s_{1}}$, we can improve the bound
\eqref{eq:62} to
\begin{align}
  \label{eq:87}
  \lvert\psi_{q}(s)(y)+e^{-\lambda_{l}s}\phi_{l}\rvert < \nu e^{-\lambda_{l}s}(y^{-\gamma}+y^{2\lambda_{l}})\qquad \text{for }Ke^{-\omega_{l}s}\le y\le e^{\sigma s}
\end{align}
for \emph{arbitrary} $\nu\in(0,1)$ provided we pick $s_{0}=s_{0}(\nu)$
large enough.

The first information that we can hope to extract is how does the
condition
\begin{align*}
  P_{s_{0},s_{1}}(q)=(\langle\psi_{q}(s_{1}),\phi_{0}\rangle,\dots,\langle\psi_{q}(s_{1}),\phi_{l-1}\rangle)=0
\end{align*}
influences the values of $q$.  For that, we need the Duhamel's formula
that we derived in Corollary~\ref{th:duhamel}, from which we get, for $n=0,1,...,\ell-1$,
\begin{align*}
0=\langle\psi_{q}( s_{1}),\phi_{n}\rangle&=e^{-\lambda_{n}( s_{1}- s_{0})}\langle\psi_{0,q},\phi_{n}\rangle+\int_{ s_{0}}^{ s_{1}}e^{-\lambda_{n}( s_{1}-s)}\langle F(\psi(s)),\phi_{n}\rangle\,ds\\
                                        &=e^{-\lambda_{n}( s_{1}- s_{0})}\left(q_{n}+e^{-\lambda_{l} s_{0}}\langle\widetilde\phi_{l},\phi_{n}\rangle\right)+\int_{ s_{0}}^{ s_{1}}e^{-\lambda_{n}( s_{1}-s)}\langle F(\psi(s)),\phi_{n}\rangle\,ds
\end{align*}
or equivalently
\begin{align}
  \label{eq:89}
  q_{n}=e^{-\lambda_{l} s_{0}}\langle\widetilde\phi_{l},\phi_{n}\rangle-\int_{ s_{0}}^{ s_{1}}e^{-\lambda_{n}( s_{0}-s)}\langle F(\psi(s)),\phi_{n}\rangle\,ds
\end{align}
with
\begin{align*}
  \widetilde\phi_{l}:=\sum_{n=0}^{l-1}e^{\lambda_{l}s_{0}}q_{n}\phi_{n}-e^{\lambda_{l}s_{0}}\psi_{0,q}.
\end{align*}
\begin{remark}
  We have defined $\widetilde\phi_{l}$ so that
  $\widetilde\phi_{l}=\phi_{l}$ in the region
  $y\in [\widetilde Ke^{\omega_{l}s_{0}},e^{\widetilde\sigma s_{0}})$.
      \end{remark}

Using the properties of $\psi_{q}$ we can now estimate the scalar
products $\langle\widetilde\phi_{l},\phi_{n}\rangle$ and
$\langle F(\psi(s)),\phi_{n}\rangle$ leading to
Lemma~\ref{lem:ball}.  As a first step towards proving
Lemma~\ref{lem:ball} we estimate the linear and nonlinear terms in
\eqref{eq:89} as follows.

\begin{Lem}
  \label{lem:linear-norm}
  The term $\widetilde \phi_{l}$ from the definition of initial data
  can be bounded by
  \begin{align*}
    \lVert\widetilde\phi_{l}-\phi_{l}\rVert\lesssim e^{-\kappa s_{0}}
  \end{align*}
  with some $\kappa>0$. Moreover, there holds
  \begin{align*}
    \lvert\langle \widetilde \phi_{l},\phi_{n}\rangle\rvert\lesssim e^{-\kappa s_{0}}
  \end{align*}
  for $n\ne l$.
\end{Lem}

\begin{proof}
  The notion of $\widetilde\phi_{l}$ is convenient since we have
  $\widetilde\phi_{l}=\phi_{l}$ in
  $y\in[\widetilde Ke^{-\omega_{l}s_{0}},e^{\widetilde\sigma s_{0}})$,
  so $\widetilde\phi_{l}$ comes very close to the pure mode
  $\phi_{l}$.  We start with
  \begin{align*}
    |\phi_l(y)-\tilde\phi_l(y)|\lesssim
    \begin{cases}
      y^{-\gamma}& \text{for } y\in(0,\widetilde Ke^{-\omega_{l} s_{0}})\\
      0 & \text{for } y\in[\widetilde Ke^{-\omega_{l} s_{0}},e^{\widetilde\sigma s_{0}})\\
      y^{2\lambda_l} & \text{for } y\in[e^{\widetilde\sigma s_{0}},\infty),
    \end{cases}
  \end{align*}
  which follows from the definition of $\widetilde\phi_{l}$ and the
  asymptotics of eigenfunctions.  The norm can now be easily estimated
  as
  \begin{align*}
    \lVert \widetilde\phi_{l}-\phi_{l}\lVert^{2}&\lesssim \int_{0}^{\widetilde Ke^{-\omega_{l} s_{0}}}y^{d-2\gamma-1}\,dy+\int_{e^{\widetilde\sigma s_{0}}}^{\infty}y^{4\lambda_{l}+d-1}e^{-\frac{y^{2}}{4}}\,dy\\
                                                      &\lesssim (Ke^{-\omega_{l} s_{0}})^{2+\omega}=e^{-(2+\omega)(1-\widetilde k)\omega_{l} s_{0}}
  \end{align*}
  (the second integral is proportional to a double exponent in $ s_0$
  and thus subdominant).  By definition $0<\widetilde k<k<1$, which
  suffices to prove the first statement of the lemma.  The second
  statement follows from the first via Cauchy-Schwartz inequality and
  orthogonality of $\phi_{l}$ and $\phi_{n}$ for $n\ne l$
$
\langle\widetilde\phi_{l},\phi_{n}\rangle=\langle\widetilde\phi_{l}-\phi_{l},\phi_{n}\rangle\le\lVert\widetilde\phi_{l}-\phi_{l} \rVert,\,\, n\ne l.
$\end{proof}

In a similar spirit we prove an estimate for the nonlinear term.

\begin{Lem}
  \label{lem:nonlinear-norm}
  For $\psi_{q}\in\mathcal W_{s_{0},s_{1}}^{1}$, $n\in\mathbb N_{0}$
  and any $s_{0}\le s\le s_{1}$ we have
    \begin{align*}
    \left\vert \langle F(\psi( s)),\phi_{n}\rangle \right\vert
\lesssim (n+1)e^{-\lambda_{l} s}e^{-\kappa\omega_{l} s}
  \end{align*}
  with $\kappa=\min(2\gamma,\omega/2)>0$ and eigenvector $\phi_{n}$.
\end{Lem}

\begin{proof}
  The Cauchy-Schwarz inequality yields
  \begin{align*}
    \lvert\langle F(\psi_{q}( s)),\phi_{n}\rangle\rvert
 \le\left(\int_{0}^{\infty}(y F(\psi( s)))^{2}y^{d-1}e^{-\frac{y^{2}}{4}}\,dy\right)^{1/2}\left\lVert\frac{\phi_{n}}{(\cdot)}\right\rVert.
  \end{align*}
  Going back to the Hardy inequality in
  Lemma~\ref{th:hardy}, in particular, to the bound \eqref{lower0} and
  pick $\alpha=(d-2)/2$, we get
  \begin{align}
    \label{eq:88}
    \lambda_{n}\lVert\phi_{n}\rVert^{2}=\langle A\phi_{n},\phi_{n}\rangle \ge \frac{\omega}{4}\left\lVert\frac{\phi_{n}}{(\cdot)}\right\rVert^{2}-\frac{d-2}{4}\lVert\phi_{n}\rVert^{2}
  \end{align}
  (the inequality was derived for
  $\phi\in C^{\infty}_{0}(\mathbb R_{+})$ but it is easy to check that
  it also holds for the eigenvectors).  Solving \eqref{eq:88} for
  $\lVert\phi_{n}/(\cdot)\rVert$ and plugging in the explicit formulae
  $\lambda_{n}=-\frac{\gamma}{2}+n$ and
  $\gamma=\frac{d-2}{2}-\frac{\omega}{2}$ we get
  \begin{align*}
    \left\lVert\frac{\phi_{n}}{(\cdot)}\right\rVert\le\left(\frac{4n}{\omega}+1\right)
  \end{align*}
  for $\phi_{n}$ normalized to $1$.

  As for the second term, by Lemma~\ref{lem:intermediate-implications}
  and some direct computations we have
  \begin{align}
    \label{eq:18}
    \lvert F(\psi(s))(y)\rvert \lesssim \frac{1}{y^{2}}
    \begin{cases}
      Q(e^{\omega_{l} s}y) & 0\le y < Ke^{-\omega_{l} s}\\
      e^{-3\lambda_{l}}(y^{-3\gamma}+y^{6\lambda_{l}}) & Ke^{-\omega_{l} s}\le y < e^{\sigma s}\\
      1 & e^{\sigma s}<y
    \end{cases}
  \end{align}
  where $Q$ is a positive function
  \begin{align}
    \label{eq:30}
    Q(\xi)=2U_{\alpha/\delta}(\xi)-2\sin(2U_{\alpha/\delta}(\xi))=c\xi^{-6\gamma}+\mathcal O(\xi^{-6\gamma-2}).
  \end{align}
  Employing the bounds \eqref{eq:18} we are led to
  \begin{align*}
    \int_{0}^{\infty}(yF(\psi( s)))^{2}y^{d-1}e^{-\frac{y^{2}}{4}}\,dy
    \lesssim\int_{0}^{Ke^{-\omega_{l} s}}Q(e^{\omega_{l} s}y)^{2}y^{d-3}e^{-\frac{y^{2}}{4}}\,dy\\
    +e^{-6\lambda_{l} s}\int_{Ke^{-\omega_{l} s}}^{1}y^{d-3-6\gamma}e^{-\frac{y^{2}}{4}}\,dy
    +e^{-6\lambda_{l} s}\int_{1}^{e^{\sigma s}}y^{d-3+6\lambda_{l}}e^{-\frac{y^{2}}{4}}\,dy\\
    +\int_{e^{\sigma s}}^{\infty}y^{d-3}e^{-\frac{y^{2}}{4}}\,dy
    =I_{1}+I_{2}+I_{3}+I_{4}.
  \end{align*}
  where we took the liberty to split the region
  $Ke^{-\omega_{l}}\le y< e^{\sigma s}$ into two and take only the
  respective dominant parts of $(y^{-3\gamma}+y^{3\lambda_{l}})$ into
  account.  In view of the rapid decay of the weight the last integral
  tends to zero with $ s\to \infty$ faster than any exponential
  function; also thanks to the weight the integral in $I_{3}$
  converges as $ s\to\infty$ hence
  \begin{align}
    \label{eq:34}
    I_{3}\lesssim e^{-2\lambda_{l} s}e^{-4\gamma\omega_{l} s}.
  \end{align}
  The weight in the remaining $I_{1}$ and $I_{2}$ can be disregarded
  thanks to $y\le 1$; in particular for $I_{1}$ this leads to
  \begin{align*}
    I_{1}\le(e^{-\omega_{l} s})^{d-2}\int_{0}^{K}Q(\xi)^{2}\xi^{d-3}\,d\xi.
  \end{align*}
  Interestingly, this integral can either converge or diverge with
  $K\to \infty$ (cf. the asymptotics of $Q$) hence we the following
  two cases
  \begin{align*}
    I_{1}\lesssim (e^{-\omega_{l} s})^{d-2}
    \begin{cases}
      K^{\omega-4\gamma} & \omega>4\gamma\\
      1 & \omega<4\gamma
    \end{cases}
  \end{align*}
  (we left out the case $\omega=4\gamma$, when the integral diverges
  logarithmically, because it corresponds to a noninteger
  dimension).  An analogous phenomenon, although in reverse, occurs
  for $I_{2}$
  \begin{align*}
    I_{2}\lesssim e^{-6\lambda_{l} s}
    \begin{cases}
      1 & \omega>4\gamma\\
      (Ke^{-\omega_{l} s})^{\omega-4\gamma} & \omega<4\gamma.
    \end{cases}
  \end{align*}
  Combining $I_{1}$ and $I_{2}$, with $2\gamma=d-2-\omega$ and
  rearranging terms we discover that
  \begin{align*}
    I_{1}+I_{2}\lesssim e^{-2\lambda_{l} s}
    \begin{cases}
      e^{-4\gamma\omega_{l} s}(1+(Ke^{-\omega_{l} s})^{\omega-4\gamma}) & \omega>4\gamma\\
      e^{-\omega\omega_{l} s}(1+K^{\omega-4\gamma}) & \omega<4\gamma.
    \end{cases}
  \end{align*}
  In view of $1\ll K$ and $Ke^{-\omega_{l} s}\ll 1$ the terms
  containing $K$ can be safely ignored as higher order
  corrections.  We can now compare the estimates for $I_{1}$, $I_{2}$
  and $I_{3}$ to conclude that
  \begin{align*}
    \left(\int_{0}^{\infty}(y F(\psi( s)))^{2}y^{d-1}e^{-\frac{y^{2}}{4}}\,dy\right)^{1/2}\lesssim e^{-\lambda_{l} s}e^{-\min(2\gamma,\omega/2)\omega_{l} s}.
  \end{align*}
The proof is now complete.
\end{proof}

\begin{Lem}
  \label{lem:ball}
  If $P_{ s_{0}, s_{1}}(q)=0$ then there exists $\kappa>0$ such that
  \begin{align*}
    \lvert q_{n}\rvert\lesssim e^{-\lambda_{l} s_{0}}e^{-\kappa s_{0}}, \qquad n=0,\dots,l-1.
  \end{align*}
\end{Lem}

\begin{proof}
  Applying the results from Lemmas~\ref{lem:linear-norm} and
  \ref{lem:nonlinear-norm} to the formula \eqref{eq:89} for
  $q_{n}$ we get
  \begin{align*}
    \lvert q_{n}\rvert& \leq
e^{-\lambda_{l}s_{0}}\lvert \langle\widetilde\phi_{l},\phi_{n}\rangle\rvert+\int_{ s_{0}}^{ s_{1}}e^{-\lambda_{n}( s_{0}-s)}\lvert\langle F(\psi_{q}(s)),\phi_{n}\rangle\rvert\,ds\\
                      &\lesssim e^{-\lambda_{l}s_{0}}\left(e^{-\kappa s_{0}}+(1+n)\int_{ s_{0}}^{ s_{1}}e^{(\lambda_{n}-\lambda_{l})s}e^{-\kappa s}\,ds\right)
  \end{align*}
  (for simplicity we took $\kappa$ to be the smaller one from the two
  Lemmas).  We know that $n\le l-1$ so we can drop the $(1+n)$
  coefficient and $\lambda_{n}-\lambda_{l}\le 1$ so the integration of
  what remains from the nonlinear term gives us
  \begin{align*}
    \lvert q_{n}\rvert\lesssim e^{-\lambda_{l}s_{0}}\left(e^{-\kappa s_{0}}+e^{-s_{0}}e^{-\kappa s_{0}}\right)\lesssim e^{-\lambda_{l}s_{0}}e^{-\kappa s_{0}},
  \end{align*}
  which proves the desired estimate.
\end{proof}

\section{A priori short time estimates}
\label{sec:short-time}
From here on, we implicitly assume that $q$ is a root of
$P_{s_{0},s_{1}}(q)=0$.  In the previous section we established that
for any root of $P_{s_{0},s_{1}}(q)=0$ there holds
$
\lvert q_{n}\rvert\lesssim e^{-\lambda s_{0}}e^{-\kappa s_{0}}
$
for some $\kappa>0$.  Here we show that for any $\nu\in(0,1)$ and
$s_{0}=s_{0}(\nu)$ large enough we can produce an improved bound of
type
\begin{align*}
\lvert \psi_{q}(s)(y)+e^{- s\lambda_l}\phi_l\rvert \le \nu e^{-\lambda_{l}s}(y^{-\gamma}+y^{2\lambda_{l}}),\qquad y\in [Ke^{\omega_{l}s},e^{\sigma s}),\qquad s-s_{0}\le 1.
\end{align*}

To get started we use a variation of constants formula to write down
the left hand side of \eqref{eq:38} as
\begin{align}
  \label{eq:varconst}
  \begin{split}
    \psi_{q}(s)+e^{- s\lambda_l}\phi_l&=e^{-A(s- s_0)}\left(\psi_{0,q}(y)+e^{- s_0\lambda_l}\phi_l\right)+\int_{ s_0}^{ s}e^{-A(s-\tau)}F(\psi(\tau))d\tau\\
    &=e^{-A(s- s_0)}\widetilde\phi_l+\int_{ s_0}^{
      s}e^{-A(s-\tau)}F(\psi(\tau))d\tau.
  \end{split}
\end{align}
Evaluated at a single point $y$ this yields
\begin{align*}
  \begin{split}
    \left|\psi(y, s)+e^{- s\lambda_l}\phi_l(y)\right|&\le \left|(e^{-A(s- s_0)}\widetilde\phi_{l})(y)\right| +\int_{ s_0}^{ s}\left
      |\left(e^{-A(s-\tau)}F(\psi(\tau))\right)(y)\right |d\tau\\
    &=S_1+S_2.
  \end{split}
\end{align*}
The next step consists of showing pointwise bounds on the action
of the semigroup $e^{-(s-s_{0})A}$ on elements of $\mathcal H$.  Via a
simple change of variables we recover the explicit heat kernel
associated to $e^{-(s-s_{0})A}$ and use the results of
Muckenhoupt~\cite{Muckenhoupt1969} to
estimate the group action by a simple maximal function.

\begin{remark}
  The estimate via the maximal function in Lemma~\ref{estim-semigroup}
  below represents a trade off.  On one hand it gives a robust
  pointwise bound on $e^{-(s-s_{0})A}\psi$.  On the other hand, the
  estimate \eqref{eq:12} contains a factor
  $(e^{-\frac{1}{2}(s-s_{0})})^{-\gamma}=e^{-\lambda_{0}(s-s_{0})}$
  growing exponentially in time (this growth rate, given by the
  smallest eigenvalue of $A$, is precisely what one would expect from
  a generic function $\psi$).  For $s-s_{0}\le 1$ the exponential
  growth boils down to a constant factor but this estimate is useless
  for long times (with a curious exception of
  Lemma~\ref{initial-data-long-time} and
  Lemma~\ref{lem:estimate-short-inhomogeneous-extension}).
\end{remark}

\newcommand{\epsz}[0]{\ensuremath{\widetilde Ke^{-\omega_l s_0}}}
\newcommand{\etaz}[0]{\ensuremath{e^{\widetilde\sigma s_0}}}

\newcommand{\epsZ}[0]{\ensuremath{Ke^{-\omega_l s_0}}}
\newcommand{\etaZ}[0]{\ensuremath{e^{\sigma s_0}}}

\begin{Lem}
  \label{estim-semigroup}
  The action of the semigroup $e^{-( s- s_0)  A}$ on a
  function $\psi(y)\in \mathcal{H}$ can be bounded as
  \begin{align}
    \label{eq:12}
    |e^{-( s- s_0)  A}\psi(y)|\lesssim (ye^{-\frac{1}{2}( s- s_{0})})^{-\gamma}(M\psi)(y),
  \end{align}
  where $M\psi$ is the maximal function defined as
  \begin{align}
    \label{eq:32}
      M\psi(y)=\sup_{I\ni
      y}\frac{\int_I(|\psi(x)|x^\gamma)x^{1+\omega}e^{-\frac{x^2}{4}}dx}{\int_I
      x^{1+\omega}e^{-\frac{x^2}{4}}dx.}
  \end{align}
  with supremum taken over all subintervals $I$ of $\mathbb{R}_{+}$ that contain $y$.

Moreover, if function $|\psi (x)|x^\gamma$ is non increasing (respectively, non decreasing),
then the supremum in \eqref{eq:32} is attained by the interval $I = [0,y]$ (respectively,
$[y,\infty)$).
\end{Lem}
\begin{proof}
  We first decompose $\psi(y)$ into eigenvectors of $A$
  \begin{align*}
    \psi(y)=\sum_{n=0}^\infty\phi_n(y)a_n,\quad a_n=\int_0^\infty\phi_n(x)\psi(x)x^{d-1}e^{-\frac{x^2}{4}}\,dx
  \end{align*}
  and then act on it with $e^{-( s- s_0)  A}$ to get
  \begin{align}
    \label{eq:111}
    e^{-( s- s_0)  A}\psi(y)&=\int_0^\infty\left(\sum_{n=0}^\infty\phi_n(x)\phi_n(y)e^{-( s- s_0)\lambda_n}\right)\psi(x)x^{d-1}e^{-\frac{x^2}{4}}\,dx.
  \end{align}
  By definition, $\phi_n$ can be written in terms of Laguerre
  polynomials, so we can rewrite the sum in parentheses as
  \begin{align}
    \label{eq:poisson-kernel}
    \begin{split}
      \sum_{n=0}^\infty\phi_n(x)\phi_n(y)e^{-( s- s_0)\lambda_n}&=2^{\omega+2}(xy)^{-\gamma}e^{( s- s_0)\frac{\gamma}{2}}\sum_{n=0}^\infty c_n^2 L_n^{(\omega/2)}\left(\frac{y^2}{4}\right)L_n^{(\omega/2)}\left(\frac{x^2}{4}\right)e^{-( s- s_0)n}\\
      &=2^{-(\omega+1)}(xy)^{-\gamma}e^{( s- s_0)\frac{\gamma}{2}}P_{\omega/2}\left(\frac{x^2}{4},\frac{y^2}{4},e^{-( s- s_0)}\right).
    \end{split}
  \end{align}
  Plugging this into \eqref{eq:111} we get
  \begin{align}
    \label{eq:40}
    \begin{split}
      e^{-( s- s_0) A}\psi(y)&=2^{-(\omega+1)}y^{-\gamma}e^{( s-
        s_0)\frac{\gamma}{2}}\int_0^\infty
      P_{\omega/2}\left(\frac{x^2}{4},\frac{y^2}{4},e^{-( s-
          s_0)}\right)\left(\psi(x)x^{-\gamma}\right)x^{d-1}e^{-\frac{x^2}{4}}\,dx\\
      &=2^{-(\omega+1)}y^{-\gamma}e^{( s-
        s_0)\frac{\gamma}{2}}\int_0^\infty
      P_{\omega/2}\left(\frac{x^2}{4},\frac{y^2}{4},e^{-( s-
          s_0)}\right)\left(\psi(x)x^{\gamma}\right)x^{\omega+1}e^{-\frac{x^2}{4}}\,dx.
    \end{split}
  \end{align}

  The positive function $P_{\omega/2}$, which replaced the sum in
  \eqref{eq:poisson-kernel}, is known as Poisson Kernel for Laguerre
  polynomials.  Muckenhoupt~\cite{Muckenhoupt1969} proved that for any
  function $f(x)\ge 0$ such that
  $f\in L^1\left(\mathbb R_+,x^{\alpha}e^{-x}\,dx\right)$, the integral
  \begin{align*}
    g(Y,z)=\int_0^\infty P_\alpha(Y,X,z)f(X)X^{\alpha}e^{-X}\,dX, \quad 0<z<1
  \end{align*}
  can be bounded by a maximal function
  \begin{align}
    \label{eq:33}
    g(Y,z)\lesssim\mathcal M f(Y),\quad \mathcal M f(Y)=\sup_{I\ni
      Y}\frac{\int_I f(X)X^\alpha e^{-X}\,dX}{\int_I
      X^\alpha e^{-X}\,dX}.
  \end{align}
  Applying \eqref{eq:33} with $\alpha=\omega/2$ to \eqref{eq:40} and
  changing variables as
  \begin{align*}
    Y=\frac{y^2}{4},\quad X=\frac{x^2}{4},\quad z=e^{-( s- s_0)},\quad f(Y)=y^{\gamma}\psi(y),
  \end{align*}
  we get
  \begin{align}
    \label{eq:60}
    |e^{-( s- s_0)  A}\psi(y)|\lesssim e^{( s- s_0)\frac{\gamma}{2}} y^{-\gamma}\cdot M\psi(y).
  \end{align}
  The new maximal operator $M$ is given by \eqref{eq:32}.
  In terms of the function $\psi$, the condition
  $f\in L^1\left(\mathbb R_+,x^{\alpha}e^{-x}\,dx\right)$, is
  equivalent to simply $|\langle\phi_{0},\psi\rangle |<\infty$, which is
  true for any $\psi\in \mathcal H$.

 The last statement of this lemma may be shown by writing the interval as $I = [a,b]$
 and differentiating with respect to $a$ or $b$.
\end{proof}

\begin{Lem}
  \label{estim-homog-term}
  For any $\nu\in(0,1)$ there exists $ s_0$ such that
  \begin{align}
    \label{eq:13}
    \lvert e^{-
      A( s- s_0)}(\psi_{0,l}+e^{-\lambda_l s_0}\phi_l)\rvert\le\nu
    e^{-\lambda_l s}(y^{-\gamma}+y^{2\lambda_l}),\qquad y\in [Ke^{-\omega_{l} s},e^{\sigma s})
  \end{align}
  for $ s-s_{0}\le1$.
\end{Lem}
\begin{proof}
  We start by noticing that
  \begin{align}
    \label{eq:44}
    \psi_{0,l}+e^{-\lambda_l s_0}\phi_l=\sum_{n=0}^{l-1}q_n\phi_n(y)+e^{-\lambda_l s_0}(\phi_l(y)-\widetilde\phi_l(y)).
  \end{align}
  We now use $ s- s_0\le1$ and $|\phi_n(y)|\le
  \alpha_{n}y^{-\gamma}+\beta_{n}y^{2\lambda_n}\lesssim y^{-\gamma}+y^{2\lambda_l}$ for $n\le l$ to bound the action of
  $e^{-( s_0- s)  A}$ on the first term in \eqref{eq:44}:
  \begin{align}
    \label{eq:91}
    \begin{split}
      \left\lvert e^{-( s- s_0)  A}\sum_{n=0}^{l-1}q_n\phi_n(y)\right\rvert&\le\left\lvert\sum_{n=0}^{l-1}e^{-( s- s_0)\lambda_n}q_n\phi_n(y)\right\rvert
\\
      &\lesssim (y^{-\gamma}+y^{2\lambda_l})\sum_{n=0}^{l-1}\lvert q_n\rvert\lesssim e^{-\lambda_l s_0}e^{-\kappa s_{0}}(y^{-\gamma}+y^{2\lambda_l}).
    \end{split}
  \end{align}
  In the last estimate (coming from Lemma~\ref{lem:ball}) $\kappa$ is
  a positive constant.  As long as $s$ and $s_0$ are comparable the
  estimate \eqref{eq:91} is of the type \eqref{eq:13}.

  The rest of this proof is devoted to estimating the semigroup action
  on the second part of \eqref{eq:44}, $\phi_l(y)-\widetilde\phi_l(y)$.
  Let us remind that
  \begin{align}
    \label{eq:57}
    |\phi_l(y)-\widetilde\phi_l(y)|\lesssim
    \begin{cases}
      y^{-\gamma}& \text{for } y\in(0,\epsz),\\
      0 & \text{for } y\in[\epsz,\etaz),\\
      y^{2\lambda_l} & \text{for } y\in[\etaz,\infty),
    \end{cases}
  \end{align}
  which we rewrite as
  \begin{align}
    \label{eq:115}
    |\phi_l(y)-\widetilde\phi_l(y)|\lesssim \mathbf
    1_{(0,\epsz)}(y)\cdot y^{-\gamma}+\mathbf
    1_{[\etaz,\infty)}(y)\cdot y^{2\lambda_l}=w_1(y)+w_2(y),
  \end{align}
  using an indicator function
  \begin{align*}
    \mathbf 1_S(x)=
    \begin{cases}
      1\quad\text{for } x\in S,\\
      0\quad\text{for } x\notin S.
    \end{cases}
  \end{align*}
  Because the heat kernel associated to $e^{-sA}$ is positive we
  have
  \begin{align}
    \label{eq:116}
    \begin{split}
      |e^{-( s- s_0)  A}(\phi_l-\widetilde\phi_l)(y)|\lesssim (e^{-( s- s_0)  A}w_1)(y)+(e^{-( s- s_0)  A}w_2)(y).
    \end{split}
  \end{align}
  Finally, we employ Lemma~\ref{estim-semigroup} to bound the both terms
  by the maximal functions:
  \begin{align}
    \label{eq:117}
    \begin{split}
      |e^{-( s- s_0)  A}(\phi_l-\widetilde\phi_l)(y)|&\lesssim y^{-\gamma}e^{\frac{\gamma}{2}( s- s_{0})}((Mw_1)(y)+(Mw_2)(y))
    \end{split}
  \end{align}
  with $M$ defined in \eqref{eq:32}.

  We proceed with the proof by exploiting the monotonicity of
  $r^{\gamma}w_1(r)$ and $r^{\gamma}w_2(r)$ to find an interval
  $I\ni y$, for which the supremum is attained.  Because
  $r^{\gamma}w_1(r)=\mathbf 1_{(0,\epsz)}(r)$ is non increasing, the
  interval $I\ni y$, for each $y>\epsz$, is simply $I=[0,y]$ and the maximal
  function evaluates to
  \begin{align}
    \label{eq:4}
    (Mw_1)(y)&=\frac{\int_{0}^{\epsz}r^{1+\omega}e^{-\frac{r^2}{4}}dr}{\int_{0}^{y}r^{1+\omega}e^{-\frac{r^2}{4}}dr}\lesssim \frac{(\epsz)^{2+\omega}}{\int_{0}^{y}r^{1+\omega}e^{-\frac{r^2}{4}}dr}.              \end{align}

  \begin{remark}
    At this point it is important to remind that from the definitions
    of $K$ and $\widetilde K$ we have
    $\widetilde Ke^{-\omega_{l}s_{0}}\ll Ke^{-\omega_{l}s}$ as long as
    $s-s_{0}\le 1$ and $s_{0}$ is sufficiently large; by assumption we
    are dealing with $Ke^{-\omega_{l}s}\le y$, so we always have
    $\epsz\ll y$.  By a similar argument we have
    $e^{\widetilde \sigma s_{0}}\ll e^{\sigma s}$.
  \end{remark}

  The denominator is minimized for the smallest admissible $y$, which
  by the assumption of this Lemma is $y=Ke^{-\omega_{l} s}$, hence
  \begin{align}
    \label{eq:86}
    (Mw_1)(y)\lesssim \left(\frac{\epsz}{Ke^{-\omega_{l} s}}\right)^{2+\omega}=\left(e^{\omega_{l}( s- s_{0})}e^{- s_{0}(k-\widetilde k)}\right)^{2+\omega}.
  \end{align}
  For short times, $ s\le s_0+1$, the first exponent is bounded, while
  the second exponent can be made arbitrarily small by making $ s_{0}$
  large thanks to $0<\widetilde k<k$.

  We continue with the second maximal function, for which
  $r^\gamma w_2(r)=\mathbf 1_{[e^{\widetilde\sigma
      s_0},\infty)}(r)r^{2l}$
  is non decreasing so, this time, the supremum of the maximal
  function is attained for $I=[y,\infty)$, yielding
  \begin{align}
    \label{eq:31}
   (Mw_2)(y)&=\frac{\int_{\etaz}^{\infty}r^{1+\omega+2l}e^{-\frac{r^2}{4}}dr}{\int_{y}^{\infty}r^{1+\omega}e^{-\frac{r^2}{4}}dr}.          \end{align}
  This time, the denominator is small for $y\gg 1$, leading to
  \begin{align}
    \label{eq:85}
    (Mw_2)(y)\lesssim y^{2l}\left(\eta^{2l+\omega+1}e^{-\frac{y^2}{4}(\eta^2-1)}(1+\mathcal O (\eta^{-2}))\right)
  \end{align}
  with
  $\eta=\etaz/y\ge e^{-\sigma( s- s_{0})}e^{-
    s_{0}(\widetilde\sigma-\sigma)}$.
  As in the case of the first maximal function, by keeping
  $ s- s_{0}$ bounded and by increasing $ s_0$, we can make $\eta$
  arbitrarily large, in turn making the quantity in parentheses
  arbitrarily small.  Combining \eqref{eq:31} and \eqref{eq:4} leads
  to $Mw_1(y)+Mw_2(y)\le \nu(1+y^{2l})$ with $\nu$ being arbitrarily
  small provided $s_{0}=s_{0}(\nu)$ is large enough.

  Combining \eqref{eq:117}, \eqref{eq:86} and \eqref{eq:85} we arrive
  at the sought for estimate
  \begin{align}
    \label{eq:79}
    |e^{-( s- s_0)  A}(\phi_l-\widetilde\phi_l)(y)|\lesssim \nu e^{\frac{\gamma}{2}( s- s_{0})}e^{-\lambda_{l} s_{0}}y^{-\gamma}(1+y^{2l})\lesssim \nu e^{-\lambda_{l} s}y^{-\gamma}(1+y^{2l})
  \end{align}
  under the assumption $ s- s_{0}\le 1$.
\end{proof}

Interestingly, by slightly restricting the choice of constants
$\sigma$ and $\widetilde\sigma$, the results of
Lemma~\ref{estim-homog-term} can be extended to long-times and large
$y$ in the sense of the following Lemma.

\begin{Lem}
  \label{initial-data-long-time}
  For any $\nu\in(0,1)$ there exists $ s_0$ such that
  \begin{align}
    \label{eq:2}
    \lvert e^{-  A( s- s_0)}(\psi_{0,l}+e^{-\lambda_l s_0}\phi_l)\rvert\le\nu e^{-\lambda_l s}(y^{-\gamma}+y^{2\lambda_l}),\qquad y\in [e^{(s-s_{0})/2},e^{\sigma s}),\,\, s \geq s_0 + 1
  \end{align}
  under the restriction
  \begin{subequations}
    \begin{align}
      \label{eq:assumption-sigma}
      \sigma < \frac{1}{4}\\
      \label{eq:assumption-tilde-sigma}
      \frac{\sigma}{1-2\sigma}<\widetilde\sigma < \frac{1}{2}.
    \end{align}
  \end{subequations}
\end{Lem}

\begin{proof}
  For the selected range of $y$ we have $y\ge 1$ so we can disregard
  the term $y^{-\gamma}$ from the right hand side of \eqref{eq:2} and
  simply show that the left hand side is bounded by
  $\nu e^{-\lambda_{l} s}y^{2\lambda_{l}}$.  For the most part of
  this proof we shall follow the steps in Lemma \ref{estim-homog-term}
  but with the assumption $1\le ye^{-( s- s_{0})/2}$ instead of
  $ s- s_{0}\le 1$.

  For example, by noticing that for $\lambda_{n}\le \lambda_{l}$
  \begin{align*}
    e^{-( s- s_{0})\lambda_{n}}y^{2\lambda_{n}}=e^{- s\lambda_{l}}y^{2\lambda_{l}}e^{ s_{0}\lambda_{l}}\left(ye^{-( s- s_{0})/2}\right)^{-2(\lambda_{l}-\lambda_{n})}\le e^{- s\lambda_{l}}y^{2\lambda_{l}}e^{ s_{0}\lambda_{l}},
  \end{align*}
  as long as $y\in [e^{(s-s_{0})/2},e^{\sigma s})$, we can modify \eqref{eq:91} in the following way:
  \begin{align*}
    \left\lvert e^{-( s- s_0)  A}\sum_{n=0}^{l-1}q_n\phi_n(y)\right\rvert
    &=\left\lvert\sum_{n=0}^{l-1}e^{-( s- s_{0})\lambda_{n}}q_{n}\phi_{n}(y)\right\rvert\\
    &\lesssim \sum_{n=0}^{l-1}|q_{n}|e^{-( s- s_{0})\lambda_{n}}y^{2\lambda_{n}}\\
    &\lesssim e^{- s\lambda_{l}}y^{2\lambda_{l}}\left(e^{ s_{0}\lambda_{l}}\sum_{n=0}^{l-1}|q_{n}|\right).
  \end{align*}
  Thanks to
  $\lvert q_{n}\rvert\lesssim e^{-\lambda_{l}s_{0}}e^{-\kappa s_{0}}$
  from Lemma~\ref{lem:ball}, the last term in parentheses can be made
  arbitrarily small by increasing $ s_{0}$.  We can use the already
  established maximal functions to produce similar bounds of the
  remaining terms.  Reusing the bound \eqref{eq:4} in the regime
  $y\in [e^{(s-s_{0})/2},e^{\sigma s})$ we get the following estimate
  for the first maximal function
  \begin{align*}
    (Mw_{1})(y)\lesssim \frac{(\widetilde Ke^{-\omega_{l} s_{0}})^{2+\omega}}{\int_{0}^{\infty}r^{1+\omega}e^{-\frac{r^2}{4}}},
  \end{align*}
  which again can be made small by the means of making $ s_{0}$
  large.

  The remaining estimate for the second maximal function requires us
  to restrict $\sigma$ and $\widetilde\sigma$.  We have, via
  \eqref{eq:85},
  \begin{align*}
    (Mw_2)(y)&\lesssim y^{2l}\left((e^{\widetilde\sigma s_{0}}/y)^{2l+\omega+1}e^{-\frac{1}{4}(e^{2\widetilde\sigma s_{0}}-y^{2})}(1+\mathcal O (\eta^{-2}))\right),
  \end{align*}
  so the smallness of the term in parentheses can only follow if
  \begin{align}
    \label{eq:29}
    y\le e^{\sigma s}\ll e^{\widetilde\sigma s_{0}}.
  \end{align}
  The region $[e^{(s-s_{0})/2},e^{\sigma s})$ is nonempty only if
  $ s- s_{0}\le 2\sigma s$, whence, in this Lemma, $ s$ is bounded by
  \begin{align}
    \label{eq:92}
     s\le \frac{ s_{0}}{1-2\sigma}.
  \end{align}
  This restriction, together with \eqref{eq:assumption-tilde-sigma},
  leads us to
  \begin{align*}
    \sigma s\le \frac{\sigma}{1-2\sigma} s_{0}<\widetilde\sigma s_{0}
  \end{align*}
  The assumption \eqref{eq:assumption-sigma} is necessary for there to
  exist a $\widetilde\sigma$ fulfilling
  \eqref{eq:assumption-tilde-sigma}.  Consequently, \eqref{eq:29}
  holds and the proof is complete.
\end{proof}

 \newcommand{\dr}[0]{\ensuremath{r^{1+\omega}e^{-\frac{r^2}{4}}dr}}

Now we turn our attention to the nonlinear term
\begin{align}
  \label{eq:42}
  \int_{ s_0}^{ s}e^{-(s-\tau) A}F(\psi(\tau))d\tau.
\end{align}
As in the previous Lemmas, we first find a suitable pointwise bound
for $|F(\psi(s))(y)|$ and then use the maximal functions to estimate
$|\left(e^{-(s-\tau) A}F(\psi(\tau))\right)(y)|$.  Our first step is
to establish a bound similar to \eqref{eq:57} but for $F(\psi(s))$.

\begin{Lem}
  \label{lem:estim-f}
  For $ s\ge s_0$, $\lambda_l>0$ and $1\ll\Gamma\le K$ we have
  \begin{align}
    \label{eq:1}
    |F(\psi( s))(y)|\lesssim e^{-\lambda_{l}  s} y^{-\gamma}
    \begin{cases}
      y^{-2}\Gamma^{\gamma}&\text{for }y<\Gamma e^{-\omega_l s}\\
      y^{-2}\Gamma^{-2\gamma}+y^{2l}e^{-(1-2\sigma)\lambda_{l} s}&\text{for }y\ge\Gamma e^{-\omega_{l} s}
    \end{cases}
  \end{align}
\end{Lem}
\begin{proof}
  From the exact form of the nonlinear term we have
  \begin{align}
    \label{eq:17}
    F(\psi(s))(y)=\frac{d-1}{2y^{2}}|\sin(2\psi( s)(y))-(2\psi( s)(y))|\lesssim y^{-2}|\psi( s)(y)|^3.
  \end{align}
  The first part of \eqref{eq:1} follows immediately from
  \eqref{hyp-inn}, that is,
  \begin{equation*}
    \left\vert \psi(s)(y) \right\vert \le
    \left\vert U_{\alpha/\delta}(ye^{\omega_{l}s})-\frac{\pi}{2} \right\vert
  \end{equation*}
  as long as $y\le \Gamma e^{-\omega_{l}s}\le K e^{-\omega_{l}s}$, and
  from $|U_{\alpha/\delta}(\xi)-\frac{\pi}{2}|\le \frac{\pi}{2}$ for
  all $\xi\ge0$, namely
  \begin{align*}
    |F(\psi(s))(y)|\lesssim y^{-2} \left|U_{\alpha/\delta}(ye^{\omega_l s})-\frac{\pi}{2}\right|^3\lesssim y^{-2}\lesssim y^{-\gamma-2}(\Gamma e^{-\omega_{l}s})^{\gamma}=y^{-\gamma-2}e^{-\lambda_{l}s}\Gamma^{\gamma}.
  \end{align*}
  The rest of the proof is devoted to the second part of
  \eqref{eq:1}.

  For the region $\Gamma e^{-\omega_l s}\le y < K e^{-\omega_l s}$,
  for which we can use \eqref{hyp-inn} again, along with \eqref{eq:17}
  and the asymptotics of $U_{\alpha/\delta}$ (see
  Lemma~\ref{th:stationary}) we get
  \begin{align}
    \label{eq:23}
    |F(\psi( s))|&\lesssim y^{-2}\left|U_{\alpha/\delta}(e^{\omega_l s}y)-\frac{\pi}{2}\right|^3\lesssim
                     y^{-2}(ye^{\omega_l s})^{-3\gamma}.
  \end{align}
  Reorganizing the terms and exploiting
  $y^{-2\gamma}\le \left( \Gamma e^{-\omega_l s} \right)^{-2\gamma} =
 \Gamma^{2\gamma} e^{2\lambda_{\ell} s}$,
  we are lead to
  \begin{align}
    \label{eq:28}
    \begin{split}
    |F(\psi( s))|&\lesssim y^{-\gamma-2}e^{-\lambda_l s}(y^{-2\gamma}e^{-2\lambda_l s})\\
    &\le y^{-\gamma-2}e^{-\lambda_l s}(\Gamma^{-2\gamma})
    \end{split}\quad\text{for }y\in[\Gamma
      e^{-\omega_l s},Ke^{-\omega_l s})
  \end{align}
  As for the intermediate region $Ke^{-\omega_l  s}\le y< e^{\sigma
     s}$, we recall \eqref{hyp-inter} and \eqref{eq:17} to obtain
  \begin{align*}
    \begin{split}
      |F(\psi( s))|      &\lesssim y^{-2}e^{-3\lambda_l  s}(y^{-3\gamma}+y^{6\lambda_l})\\
      &=e^{-\lambda_{l} s}y^{-\gamma}(e^{-2\lambda_l  s}y^{-2\gamma-2})+e^{-\lambda_{l} s}y^{2\lambda_{l}}(e^{-2\lambda_{l} s}y^{4\lambda_l-2})\\
      &\le e^{-\lambda_{l} s}y^{-\gamma-2}(K^{-2\gamma})+e^{-\lambda_{l} s}y^{2\lambda_{l}}(e^{-2(1-2\sigma)\lambda_{l} s -2\sigma s})\\
      &\le e^{-\lambda_{l} s}y^{-\gamma-2}(\Gamma^{-2\gamma})+e^{-\lambda_{l} s}y^{2\lambda_{l}}(e^{-(1-2\sigma)\lambda_{l} s -2\sigma s)})
    \end{split}
  \end{align*}
  In the last line we replace $K$ with $\Gamma$ by means of
  $\Gamma\le K$ and we drop the '$2$' from the second summand so that
  it agrees with the following estimate for the external region.  For
  $y\ge e^{\sigma s}$, by the virtue of \eqref{hyp-out} and
  \eqref{eq:17}, we have
  \begin{align*}
    \begin{split}
      |F(\psi( s))|\lesssim y^{-2} &= e^{-\lambda_l s}y^{2\lambda_{l}}(y^{-2\lambda_{l}-2}e^{\lambda_l s})\\
      &\le e^{-\lambda_l s}y^{2\lambda_{l}}(e^{-(1-2\sigma)\lambda_{l} s-2\sigma s}),
    \end{split}
  \end{align*}
which completes the proof.
\end{proof}

\begin{Lem}
  \label{lem:estimate-short-inhomogeneous}
  For any $\nu\in(0,1)$ there exists $s_0$ such that
  \begin{align*}
    \left\lvert\int_{ s_0}^{ s}e^{-(s-\tau)A}F(\psi(\tau))ds\right\rvert\le \nu e^{- s\lambda_l}(y^{-\gamma}+y^{2\lambda_{l}}),\qquad y\in[Ke^{-\omega_{l}s},e^{\sigma s})
\, \text{ and } s- s_0\le 1.
  \end{align*}
\end{Lem}
\begin{proof}
  The proof uses a similar technique as the homogeneous case
  (Lemma~\ref{estim-homog-term}).  Namely, we estimate the nonlinear
  term according to Lemma~\ref{lem:estim-f} and then apply the
  Lemma~\ref{estim-semigroup}.  Let us start by rewriting the estimate
  from Lemma~\ref{lem:estim-f} as
  \begin{align*}
    |F(\psi(\tau))|\lesssim e^{-\lambda_{l} \tau}y^{-\gamma}(f_{1}(y)+f_{2}(y)+f_{3}(y))
  \end{align*}
  with
  \begin{align*}
    f_{1}(y)&=\mathbf 1_{[0,\Gamma e^{-\omega_{l}\tau})}(y)\, y^{-2}\Gamma^{\gamma},\\
    f_{2}(y)&=\mathbf 1_{[\Gamma e^{-\omega_{l}\tau},\infty)}(y)\,y^{-2}\Gamma^{-2\gamma}\\
    f_{3}(y)&=\mathbf 1_{[\Gamma e^{-\omega_{l}\tau},\infty)}(y)\,y^{2l}e^{-(1-2\sigma)\lambda_{l}\tau}.
  \end{align*}
  By means of Lemma~\ref{estim-semigroup} we have
  \begin{align}
    \label{eq:104}
    |e^{-(s-\tau) A}F(\psi(\tau))|\lesssim e^{(s-\tau)\frac{\gamma}{2}}y^{-\gamma}e^{-\lambda_{l}\tau}(Mf_{1}(y)+Mf_{2}(y)+Mf_{3}(y)),
  \end{align}
  so it is enough to show that each maximal function can be made much
  smaller than $1+y^{2l}$ with appropriate choice of $s_{0}$.  We
  shall expect that $f_{1}$ and $f_{2}$ have maximal functions that
  dominate for small $y$, while the maximal function for $f_{3}$
  dominates the large $y$ region.

  As for $Mf_{1}(y)$, the function $f_{1}$ is nonincreasing, so the
  supremum in maximal function is attained for the interval $[0,y)$,
  hence
  \begin{align*}
    Mf_{1}(y)=\Gamma^{\gamma}\frac{\int_{0}^{\Gamma e^{-\omega_{l}\tau}}r^{\omega-1}e^{-r^{2}/4}dr}{\int_{0}^{y}r^{\omega+1}e^{-r^{2}/4}dr}.
  \end{align*}
  The numerator is integrable thanks to $\omega>0$. The denominator
  is then minimized for $y=Ke^{-\omega_{l}\tau}$. Moreover, we can skip the
  exponential terms as near the origin they are of order one; these
  simplifications lead us to
  \begin{align}
    \label{eq:16}
    Mf_{1}(y)\lesssim\Gamma^{\gamma}\frac{(\Gamma e^{-\omega_{l}\tau})^{\omega}}{(Ke^{-\omega_{l}\tau})^{\omega+2}}=\Gamma^{\gamma+\omega}K^{-\omega-2}e^{2\omega_{l}\tau}=e^{\omega_{l}s_{0}(k{\theta}(\gamma+\omega)-k(\omega+2)+2)}e^{2\omega_{l}(\tau-s_{0})},
  \end{align}
  where, for convenience, we defined $\Gamma$ as a power of $K$:
  \begin{align*}
    \Gamma:=K^{\theta},\qquad \theta \in(0,1)
  \end{align*}
  and used the definition $K=e^{k\omega_{l}s_{0}}$.  The last term in
  \eqref{eq:16} decays exponentially with $s_{0}$ if we pick $\theta$
  such that
  \begin{align}
    \label{eq:96}
    \theta < \frac{k(\omega+2)-2}{k(\gamma+\omega)}.
  \end{align}

  In a similar fashion we estimate $Mf_{2}(y)$, although this time, we
  extend the support of $f_{2}$ by dropping the indicator function
  (i.e. exploiting $\mathbf 1_S\le 1$).  Because $f_{2}$ without the
  indicator function is nonincreasing, the explicit interval for the
  maximal function is $[0,y)$ and we have
  \begin{equation*}
    Mf_{2}(y)\lesssim \Gamma^{-2\gamma}\frac{\int_{0}^{y}r^{\omega-1}e^{-r^{2}/4}dr}{\int_{0}^{y}r^{\omega+1}e^{-r^{2}/4}dr}.
  \end{equation*}
  As expected, the maximal function attains its maximum at
  $y=Ke^{-\omega_{l}s}$, so once again we can drop the exponential
  factors and reduce $Mf_{2}$ to two simple integrals
  \begin{equation*}
    Mf_{2}(y)\lesssim \Gamma^{-2\gamma}\frac{(Ke^{-\omega_{l}s})^{\omega}}{(Ke^{-\omega_{l}s})^{\omega+2}}=e^{2\omega_{l}\tau_{0}(1-k-k\gamma\theta)}e^{2\omega_{l}(s-\tau_{0})}.
  \end{equation*}
  This time, for the right hand side to be decreasing to $0$ with $\tau_{0}$,
  we must have
  \begin{align}
    \label{eq:90}
    \frac{(1-k)}{k\gamma}<\theta
  \end{align}
  but at the same time condition \eqref{eq:96} must be satisfied as
  well, leading to the following condition for $k$:
  \begin{align*}
    \frac{(1-k)}{k\gamma} < \frac{k(\omega+2)-2}{k(\gamma+\omega)},
  \end{align*}
  or equivalently
  \begin{align*}
    k_{0}:=\frac{3\gamma+\omega}{3\gamma+\omega+\gamma\omega} < k < 1.
  \end{align*}
  It is easy to see that $k_{0}<1$ (that is, if $\gamma,\omega>0$,
  which is always true for $d>4+2\sqrt{2}$), so it is always possible
  to choose such $k$.

  The last term, $Mf_{3}$, is the easiest one to control.
  Since $f_{3}$ is an increasing function, its maximal function reduces to
  \begin{align*}
    Mf_{3}(y)=e^{-(1-2\sigma)\lambda_{l}\tau}\frac{\int_{y}^{\infty}r^{2l+1+\omega}e^{-r^{2}/4}dr}{\int_{y}^{\infty}r^{1+\omega}e^{-r^{2}/4}dr}.
  \end{align*}
  By standard results for gamma functions the above ratio can be
  bounded by
  \begin{align*}
    Mf_{3}(y)\lesssim e^{-(1-2\sigma)\lambda_{l}\tau}(1+y^{2l}),
  \end{align*}
  which already has a right decay with $\tau$ (or equivalently with
  $s_{0}$).

  It now remains to integrate the formula \eqref{eq:104} over
  $\tau\in[s_{0},s]$, taking into account $s-s_{0}\le 1$, to arrive at
  the thesis of this lemma.
\end{proof}

We follow up with the extension of Lemma
\ref{lem:estimate-short-inhomogeneous} to long times.

\begin{Lem}
  \label{lem:estimate-short-inhomogeneous-extension}
  For any $\nu\in(0,1)$ there exists $s_0$
  such that
  \begin{align*}
    \left\lvert\int_{s_0}^{s}e^{-(s-\tau) A}F(\psi(\tau))ds\right\rvert\le \nu e^{-s\lambda_l}y^{-\gamma},\qquad y\in [e^{(s-s_{0})/2},e^{\sigma s}), \,\,
s \geq s_0 + 1
  \end{align*}
  provided that
  \begin{align*}
    \sigma<\frac{1}{10l}.
  \end{align*}
\end{Lem}

\begin{proof}
  The proof follows from the same ideas as of
  Lemma~\ref{initial-data-long-time} by using estimates already
  established in Lemma~\ref{lem:estimate-short-inhomogeneous}.  We
  first note that for $y\ge e^{(s-s_{0})/2}$ we have
  $y^{-\gamma}\le y^{2\lambda_{l}}e^{-l(s-s_{0})}$ and combine it with
  \eqref{eq:104} to get
  \begin{align*}
    |e^{-(s-\tau) A}F(\psi(\tau))|&\lesssim e^{(s-\tau)\frac{\gamma}{2}}y^{-\gamma}e^{-\lambda_{l}\tau}(Mf_{1}(y)+Mf_{2}(y)+Mf_{3}(y))\\
    &\le e^{-\lambda_{l}s}y^{2\lambda_{l}}e^{-l(\tau-s_{0})}(Mf_{1}(y)+Mf_{2}(y)+Mf_{3}(y)).
  \end{align*}
  We start with
  \begin{align*}
    Mf_{1}(y)&=\Gamma^{\gamma}\frac{\int_{0}^{\Gamma e^{-\omega_{l}\tau}}r^{\omega-1}e^{-r^{2}/4}dr}{\int_{0}^{y}r^{\omega+1}e^{-r^{2}/4}dr}\\
    &\lesssim \Gamma^{\gamma}\int_{0}^{\Gamma e^{-\omega_{l}\tau}}r^{\omega-1}e^{-r^{2}/4}dr
  \end{align*}
  as long as $y\ge e^{(s-s_{0})/2}\ge 1$.  As before, we skip the
  exponential factor and approximate the last integral by
  $(\Gamma e^{-\omega_{l}\tau})^{\omega}$ to arrive at
  \begin{align*}
    Mf_{1}(y) \lesssim \Gamma^{\omega+\gamma}e^{-\omega\omega_{l}\tau}\le\Gamma^{\omega+\gamma}e^{-\omega\omega_{l}s_{0}}=e^{\omega_{l}s_{0}(k\theta(\gamma+\omega)-\omega)},
  \end{align*}
  which decays with $s_{0}$ as long as
  \begin{align*}
    \theta<\frac{\omega}{k(\gamma+\omega)}.
  \end{align*}
  This new condition for $\theta$ is superfluous to \eqref{eq:96}
  where we already restricted $\theta$ by
  \begin{align*}
    \theta<\frac{k(\omega+2)-2}{k(\gamma+\omega)}<\frac{\omega}{k(\gamma+\omega)}.
  \end{align*}
  Next comes $Mf_{2}$, for which we have
  \begin{align*}
    Mf_{2}(y)&\lesssim \Gamma^{-2\gamma}\frac{\int_{0}^{y}r^{\omega-1}e^{-r^{2}/4}dr}{\int_{0}^{y}r^{\omega+1}e^{-r^{2}/4}dr}\\
    &\lesssim \Gamma^{-2\gamma}(1+\mathcal O(y^{-2}))=e^{-2\gamma\omega_{l}k\theta}(1+\mathcal O(y^{-2}))
  \end{align*}
  as long as $y\ge e^{(s-s_{0})/2}\ge 1$ and the decay with $s_{0}$ is
  straight forward.  For the last maximal function we get
  \begin{align*}
    Mf_{3}(y)&=e^{-(1-2\sigma)\lambda_{l}s}\frac{\int_{y}^{\infty}r^{2l+1+\omega}e^{-r^{2}/4}dr}{\int_{y}^{\infty}r^{1+\omega}e^{-r^{2}/4}dr}\\
    &\lesssim y^{2l}e^{-(1-2\sigma)\lambda_{l}s}
 \le y^{2l}e^{-(1-2\sigma) \lambda_{l} s_{0}}\le e^{2l\sigma s}e^{-(1-2\sigma) \lambda_{l} s_{0}}.
  \end{align*}
  At this point we reuse \eqref{eq:90} from
  Lemma~\ref{initial-data-long-time}, which puts a limit on $s$ in
  terms of $s_{0}$ to get
  \begin{align*}
    Mf_{3}(y)\lesssim e^{\left(\frac{2l\sigma}{1-2\sigma}-(1-2\sigma)\right)s_{0}}.
  \end{align*}
  At this point, we have to impose an additional condition on
  $\sigma$, namely $Mf_{3}$ decays with $s_{0}$ if
  \begin{align*}
    0<\sigma <\frac{1}{10l}.
  \end{align*}
  Again, this is compatible with the assumption $\sigma<\frac{1}{4}$
  we made in Lemma~\ref{initial-data-long-time}, in fact the latter
  assumption is superfluous because
  $0<\sigma <\frac{1}{10l} < \frac{1}{4}$ for all $l\ge1$.
\end{proof}

\section{A priori long time estimates}
\label{sec:long-time}
\begin{Lem}
  \label{lem:long-bounded}
  If $\psi\in\mathcal W_{ s_{0}, s_{1}}^{1}$ and
  $P_{ s_{0}, s_{1}}(q_{n})=0$ with $n = 0,1,...,\ell -1$, then for any $R>1$
there exists $\kappa>0$ such that
  \begin{align*}
    \lvert\psi( s)(y)+e^{-\lambda_{l} s}\phi_{l}(y)\rvert \lesssim C(R)e^{-\kappa s_{0}}e^{-\lambda_{l} s}y^{-\gamma}
  \end{align*}
  for $Ke^{-\omega_{l} s}< y\le R$ and $s> s_{0}+1$.
\end{Lem}
\begin{proof}
  First, let us remind that $\psi( s)+e^{-\lambda_{l} s}\phi_{l}$ solves
  \begin{align}
    \label{eq:3}
    \psi( s)+e^{-\lambda_{l} s}\phi_{l}=e^{- A( s- s_{0})}(\psi_{0}+e^{-\lambda_{l} s_{0}}\phi_{l})+\int_{ s_{0}}^{ s}e^{- A( s-\tau)}f(\psi(\tau))\,d\tau.
  \end{align}
  It is convenient to divide the nonlinear part into the following two integrals:
  \begin{align*}
    \int_{ s_{0}}^{ s}e^{- A( s-\tau)}f(\psi(\tau))\,d\tau=\left(\int_{ s_{0}}^{ s-1}+\int_{ s-1}^{ s}\right)e^{- A( s-\tau)}f(\psi(\tau))\,d\tau
  \end{align*}
  By arguments similar to ones used for the short-time estimates (see
  the proof of Lemma~\ref{lem:estimate-short-inhomogeneous}), the
  integral over $[ s-1, s]$, denoted from now on as $\mathcal I$,
  already fulfills the proposed bounds, so we shall only consider the
  remaining integral over $[ s_{0}, s-1]$.
  Let us now rewrite the right hand side of \eqref{eq:3} in terms of
  projections onto eigenfunctions
  \begin{align*}
    \psi( s)+e^{-\lambda_{l} s}\phi_{l}=\sum_{n=0}^{\infty}\phi_{n}\left(e^{-\lambda_{n}( s- s_{0})}\langle\psi_{0}+e^{-\lambda_{l} s_{0}}\phi_{l},\phi_{n}\rangle+\int_{ s_{0}}^{s-1}e^{-\lambda_{n}(s-\tau)}\langle f(\psi(\tau)),\phi_{n}\rangle d\tau \right)+\mathcal I.
  \end{align*}
  Exploiting the particular form of $\psi_{0}$ and the fact that
  $P_{ s_{0}, s_{1}}(q_{n})=0$, we get
  \begin{subequations}
    \label{eq:41}
    \begin{align}
      \label{eq:5}
      \psi( s)+e^{-\lambda_{l} s}\phi_{l}&=-\sum_{n=0}^{l-1}\phi_{n}\int_{ s-1}^{ s_{1}}e^{-\lambda_{n}(s-\tau)}\langle f(\psi(\tau)),\phi_{n}\rangle\,d\tau\\
      \label{eq:45}
                                         &-\phi_{l}e^{-\lambda_{l} s}\langle\widetilde\phi_{l}-\phi_{l},\phi_{l}\rangle\\
      \label{eq:8}
                                         &+\sum_{n=l+1}^{\infty}\phi_{n}e^{-\lambda_{n}( s- s_{0})}e^{-\lambda_{l} s_{0}}\langle\widetilde\phi_{l}-\phi_{l},\phi_{n}\rangle\\
      \label{eq:10}
                                         &+\sum_{n=l+1}^{\infty}\phi_{n}\int_{ s_{0}}^{ s-1}e^{-\lambda_{n}(s-\tau)}\langle f(\psi(\tau)),\phi_{n}\rangle\,d\tau\\
      \label{eq:9}
                                         &+\mathcal I.
    \end{align}
  \end{subequations}
  Lemmas~\ref{lem:nonlinear-norm} and \ref{lem:linear-norm}, along
  with the asymptotic behavior for eigenfunctions from
  Lemma~\ref{th:extension} imply the appropriate bounds for the terms
  \eqref{eq:5} and \eqref{eq:45}; for example there holds
  \begin{align*}
    \left\vert \phi_{n}\int_{ s-1}^{ s_{1}}e^{-\lambda_{n}(s-\tau)}\langle f(\psi(\tau)),\phi_{n}\rangle\,d\tau\right\vert
& \lesssim e^{-\lambda_{l}s}y^{-\gamma}\alpha_{n}\int_{ s-1}^{ s_{1}}e^{-(\lambda_{l}-\lambda_{n})(\tau-s)}e^{-\kappa \tau}\,d\tau \\
& \lesssim e^{-\kappa s}e^{-\lambda_{l} s}y^{-\gamma}\alpha_{n}e^{(\lambda_{l}-\lambda_{n})},
  \end{align*}
  where we have used the fact that $\lambda_{l}-\lambda_{n}>0$ for
  $n\le l-1$.
 The terms depending on $n$ can be safely bounded by a
  single constant because the summation is over a finite range of $n$.

  The third and fourth terms, \eqref{eq:8} and \eqref{eq:10}, must be
  treated more carefuly---since the summation goes to infinity and
  thus an inefficient bound on the projections may simply
  diverge.  As we shall see, the results proved in
  Lemmata~\ref{lem:nonlinear-norm} and \ref{lem:linear-norm} are
  sufficient as they lead to
  \begin{equation}
    \label{eq:35}
    \begin{split}
      \lvert\eqref{eq:8}\rvert&\lesssim y^{-\gamma}\sum_{n=l+1}^{\infty}\alpha_{n}e^{-\lambda_{n}( s- s_{0})}e^{-\lambda_{l} s_{0}}e^{-\kappa s_{0}}\\
      &= e^{-\kappa s_{0}}e^{-\lambda_{l} s}y^{-\gamma}\sum_{n=l+1}^{\infty}\alpha_{n}e^{-(\lambda_{n}-\lambda_{l})( s- s_{0})}
    \end{split}
  \end{equation}
  and
  \begin{equation}
    \label{eq:6}
    \begin{split}
      \lvert\eqref{eq:10}\rvert&\lesssim y^{-\gamma}\sum_{n=l+1}^{\infty}\alpha_{n}\lambda_{n}\int_{ s_{0}}^{ s-1}e^{-\lambda_{n}(s-\tau)}e^{-\lambda_{l}\tau}e^{-\kappa \tau}\,d\tau\\
      &\lesssim e^{-\kappa s}e^{-\lambda_{l} s}y^{-\gamma}\sum_{n=l+1}^{\infty}\frac{\alpha_{n}\lambda_{n}}{\lambda_{n}-\lambda_{l}-\kappa}e^{-(\lambda_{n}-\lambda_{l})}.
    \end{split}
  \end{equation}
  The leading order coefficient, $\alpha_{n}$, grows with $n$, but
  only algebraically, which is easily countered by the exponential
  decay of $e^{-(\lambda_{n}-\lambda_{l})}$ as long as
  $ s- s_{0}\ge1$; therefore the sums in \eqref{eq:35} and
  \eqref{eq:6} converge and the proof is complete.
\end{proof}

\begin{Lem}
  \label{lem:long-unbounded}
  For any $\nu\in(0,1)$ there exists $R$ and $s_{0}$ large enough so
  that
  \begin{align*}
    \lvert\psi_{q}(s)(y)+e^{-\lambda_{l} s}\phi_{l}(y)\rvert \lesssim \nu e^{-\lambda_{l} s}y^{2\lambda_{l}}, \qquad R<y\le e^{(s- s_{0})/2}
  \end{align*}
  as long as $s>s_{0}+1$.
\end{Lem}

\begin{proof}
  Let us now fix $y>R$ and an intermediate time
  $\bar s\in [s_{0},s-1]$, such that
  $e^{(s-\bar s)/2}\le y\le 2 e^{(s-\bar s)/2}$ (for any
  $y\in [R,e^{(s-s_{0})/2})$ it is possible to find such $\bar s$) and
  let us write $\psi_{q}(s)$ relative to $\psi_{q}(\bar s)$ as
  \begin{align*}
    \psi_{q}(s)+e^{-\lambda_{l}s}\phi_{l}=e^{-A(s-\bar s)}(\psi_{q}(\bar s)+e^{-\lambda_{l}\bar s}\phi_{l})+\int_{\bar s}^{s}e^{-A(s-\tau)}F(\psi_{q}(\tau))\,d\tau.
  \end{align*}
  In other words, we are using $\psi_{q}(\bar s)$ as an initial data to
  get $\psi_{q}(s)$ for $s\ge \bar s+1$.
  We then notice that the
  nonlinear term in the above expression can be bounded by the same
  logic as what led us to
  Lemma~\ref{lem:estimate-short-inhomogeneous-extension}, once we replace
  $s_{0}$ there by $\bar s$.
As a result, we obtain
  \begin{align*}
    \left\lvert\int_{\bar s}^{s}e^{-A(s-\tau)}F(\psi_{q}(\tau))(y)\,d\tau\right\rvert\lesssim e^{-\kappa \bar s} e^{-s\lambda_{l}}y^{-\gamma},\qquad y\in[e^{(s-\bar s)/2},e^{\sigma s}).
  \end{align*}
  Because $e^{-\kappa \bar s}\le e^{-\kappa s_{0}}$ we have just
  produced the required bound for the nonlinear term.

  As for the linear term, the previous Lemma, combined with the
  starting assumption on $\psi_{q}$ guarantees that
  \begin{align}
    \label{eq:19}
    \lvert\psi_{q}(s)(y)+e^{-\lambda_{l}s}\phi_{l}(y)\rvert\lesssim e^{-\lambda_{l}s}
    \begin{cases}
      e^{-\kappa s_{0}}C(R)y^{-\gamma}& y\le R,\\
      y^{2\lambda_{l}}&y>R.
    \end{cases}
  \end{align}
  In other words, we start with a bound that is already improved for
  $y\le R$.  Now observe that \eqref{eq:19} implies
  \begin{align*}
    \lVert\psi_{q}(s)+e^{-\lambda_{l}s}\phi_{l}\rVert^{2}&\lesssim e^{-2\lambda_{l}s}e^{-2\kappa s_{0}}C(R)^{2}\int_{0}^{R}y^{-2\gamma+d-1}\,dy+e^{-2\lambda_{l}s}\int_{R}^{\infty}y^{4\lambda_{l}+d-1}e^{-\frac{y^{2}}{4}}\,dy\\
                                                         &\lesssim e^{-2\lambda_{l}s}(e^{-2\kappa s_{0}}C_{1}(R)+C_{2}(R)).
  \end{align*}
  From the form of the integrals one can easily see that $C_{1}$
  increases with $R$, while $C_{2}$ tends to zero with $R$.  Thus for
  every $\nu\in (0,1)$ there exists $R=R(\nu)$ and $s_{0}=s_{0}(R)$,
  both large enough, so that the sum
  $e^{-2\kappa s_{0}}C_{1}(R)+C_{2}(R)$ can be made smaller than
  $\nu$, therefore
  \begin{align*}
    \lvert\langle \psi_{q}(s)+e^{-\lambda_{l}s}\phi_{l},\phi_{n}\rangle\rvert\le \lVert\psi_{q}(s)+e^{-\lambda_{l}s}\phi_{l}\rVert\lesssim \nu e^{-\lambda_{l}s}.
  \end{align*}
  We immediately get
  \begin{align*}
    \lvert e^{-A(s-\bar s)}(\psi_{q}(\bar s)+e^{-\lambda_{l}\bar s}\phi_{l})(y)\rvert
&\le\sum_{n=0}^{\infty}e^{-\lambda_{n}(s-\bar s)}\lvert\langle \psi_{q}(\bar{s}) +
e^{-\lambda_{l}\bar{s}}\phi_{l},\phi_{n}\rangle\rvert\lvert\phi_{n}(y)\rvert\\
                                                                                     &\lesssim \nu\sum_{n=0}^{\infty} e^{-\lambda_{n}(s-\bar s)}e^{-\lambda_{l}\bar s} \beta_{n}y^{2\lambda_{n}}\\
                                                                                     &=\nu e^{-\lambda_{l}s}y^{2\lambda_{l}}\sum_{n=0}^{\infty}\beta_{n}(ye^{-(s-\bar s)/2})^{2\lambda_{n}-2\lambda_{l}}.
  \end{align*}
  It now suffices to show that the last sum converges and can be
  bounded independently of $s$.  But this becomes evident from our
  choice of $\bar s$: we picked $\bar s$ such that
  $1\le y e^{-(s-\bar s)/2}\le 2$.  We simply split the sum into two
  components containing either negative and positive powers of
  $(ye^{-(s-\bar s)/2})$ to get
  \begin{align*}
    \sum_{n=0}^{\infty}\beta_{n}(ye^{-(s-\bar s)/2})^{2\lambda_{n}-2\lambda_{l}}&=\sum_{n=0}^{l}\beta_{n}(ye^{-(s-\bar s)/2})^{2\lambda_{n}-2\lambda_{l}}+\sum_{n=l+1}^{\infty}\beta_{n}(ye^{-(s-\bar s)/2})^{2\lambda_{n}-2\lambda_{l}}\\
                                                                                &\le \sum_{n=0}^{l}\beta_{n}+\sum_{n=l+1}^{\infty}\beta_{n}2^{2\lambda_{n}-2\lambda_{l}}.
  \end{align*}
  Thanks to $\beta_{n}$ behaving roughly as $n^{-\omega/4}4^{-n}/n!$ the
  last sum converges, while the first sum is is simply finite.
\end{proof}

\noindent
\textbf{Acknowledgments}.
The first author acknowledges financial support by the Sofja
Kovalevskaja award of the Humboldt Foundation endowed by the German
Federal Ministry of Education and Research held by Roland
Donninger.  Partial support by the DFG, SFB 1060 is also gratefully
acknowledged.  The second author was partly supported by Grant-in-Aid
for Research Activity Start-up (No. 6887027) and Kyushu University
Interdisciplinary Programs in Education and Projects in Research
Development (No. 2012DB4000).

\bibliographystyle{abbrvnat}
\bibliography{/home/pawel/Documents/Mendeley/library}

\end{document}